\documentclass[11pt,a4paper]{article}
\usepackage[T1]{fontenc}
\usepackage[utf8]{inputenc}
\usepackage[margin=27mm]{geometry}
\usepackage{lmodern,amsmath,amssymb,amsthm,mathtools,microtype}
\usepackage{graphicx,tikz,enumitem,float,xcolor,flafter}
\usepackage[font=small,labelfont=bf]{caption}
\usepackage{algorithm,algpseudocode,booktabs,adjustbox}
\usepackage[section]{placeins}

\usepackage[hidelinks]{hyperref}
\usepackage{orcidlink}
\usetikzlibrary{arrows.meta,graphs,shapes.geometric,shapes.misc,decorations.pathreplacing,positioning}
\definecolor{oiblue}{HTML}{0072B2}
\definecolor{oiverm}{HTML}{B34700}
\definecolor{oigreen}{HTML}{007C59}
\definecolor{oipurple}{HTML}{A25B8A}
\definecolor{oiorange}{HTML}{A66B00}
\numberwithin{equation}{section}
\newtheorem{theorem}{Theorem}[section]
\newtheorem{lemma}[theorem]{Lemma}
\newtheorem{proposition}[theorem]{Proposition}
\newtheorem{corollary}[theorem]{Corollary}
\theoremstyle{remark}
\newtheorem{remark}[theorem]{Remark}
\newcommand{\NN}{\mathcal N_k}
\newcommand{\PP}{\mathcal P_k}
\newcommand{\BB}{\mathcal B_k}
\newcommand{\edge}[2]{\{#1,#2\}}

\newcommand{\intset}[2]{\{#1,\ldots,#2\}}

\hypersetup{pdftitle={Minimum blockers and extremal graphs for nonnested matchings},pdfauthor={Pedro M. M. de Castro},pdfsubject={Minimum blockers and extremal ordered graphs},pdfkeywords={ordered graph; nonnested matching; minimum blocker; extremal graph}}
\title{Minimum blockers and extremal graphs for nonnested matchings}
\author{Pedro M. M. de Castro\,\orcidlink{0000-0002-9470-7458}\\
\small Centro de Inform\'atica, Universidade Federal de Pernambuco\\
\small Recife, Pernambuco, Brazil\\
\small \texttt{pmmc@cin.ufpe.br}}
\date{October 7, 2026}
\begin{document}
\maketitle
\begin{abstract}
We classify all smallest sets of edge deletions that destroy every nonnested perfect matching in a complete ordered graph. The vertices have a fixed linear order. A perfect matching selects edges that use each vertex exactly once; it is nonnested when no selected edge has both endpoints strictly between those of another. On $2k$ vertices, exactly $k$ deletions are necessary, and we describe all $2^k+k-2$ minimum deletion sets for every $k\ge2$.

Allowing a matching to leave vertices unused changes the extremal problem. Barát, Freschi and Tóth proposed that an $n$-vertex ordered graph avoiding a nonnested $k$-edge matching can have at most $(k-1)n$ edges. We construct counterexamples for every $k\ge5$ and $n\ge2k+1$. For $n\ge3k$, our constructions exceed the proposed value by a number of edges proportional to $k^2$, matching the order of the known upper bound on this excess. The classification follows from cuts between consecutive intervals; the larger constructions coordinate deletions at the two ends of the vertex order. The exact extremal value remains open in general.
\end{abstract}

\medskip
\noindent\textbf{AI-use disclosure.}
OpenAI's ChatGPT and Codex, together with Anthropic's Claude, were used as research tools to assist the author with literature retrieval, mathematical exploration, symbolic and numerical checks, and draft preparation and revision. The author formulated the research questions and mathematical framework, developed and refined the results and arguments with this assistance, and critically reviewed the manuscript throughout. Responsibility for the mathematical statements and the final manuscript rests entirely with the author.

\medskip
\noindent\textbf{Keywords.} ordered graph; nonnested matching; minimum blocker; extremal graph.

\section{Introduction}\label{sec:introduction}
Selecting disjoint edges from a graph means pairing vertices while using each vertex at most once. Such a selection is called a \emph{matching}. The graph can have other edges between the selected vertices, and these edges remain available for other choices of matching. For example, in a graph containing all six edges between four vertices, we may select the two edges $\{1,2\}$ and $\{3,4\}$. We study how large a matching can be when the vertices have a fixed linear order and the selected pairs must respect an additional restriction.

The restriction forbids \emph{nesting}: both endpoints of one selected edge may not lie strictly between the endpoints of another. Thus, in the order $1<2<3<4$, selecting $\{1,4\}$ together with $\{2,3\}$ is forbidden, whereas $\{1,3\}$ together with $\{2,4\}$ is allowed. A large ordinary matching can fail this condition badly: on six vertices, $\{1,6\}$, $\{2,5\}$ and $\{3,4\}$ form a matching, yet every two of them are nested. Even finding two permitted edges can therefore require a different choice.

We ask how many edges must be deleted from a complete ordered graph to prevent every nonnested matching of a given size. Equivalently, how many edges can remain while avoiding every such selection? When the matching must use all vertices, we classify every smallest set of deletions. When some vertices may be left unused, we construct graphs that exceed a proposed extremal value by a quadratic term in the matching size. We develop these two regimes through small examples before stating their general forms.

\subsection{The matchings we seek}\label{sec:basic-objects}
A matching is a choice made within a given graph, so we first distinguish the available edges from the selected ones. Formally, a matching $M$ in a graph $G$ is a subset of its edge set such that each vertex belongs to at most one edge of $M$. Fig.~\ref{fig:matching-selection} shows two choices in the same graph. The extra edges of $G$ are retained in both drawings.

\begin{figure}[htbp]
\centering
\begin{tikzpicture}[font=\small,
 v/.style={circle,draw=black,fill=white,minimum size=5mm,inner sep=0pt},
 bg/.style={gray!55,thin}]
\foreach \shift/\letter in {0/a,7.2/b}{
 \begin{scope}[xshift=\shift cm]
 \coordinate (a) at (0,2); \coordinate (b) at (3,2);
 \coordinate (c) at (0,0); \coordinate (d) at (3,0);
 \draw[bg] (a)--(b) (a)--(c) (a)--(d) (b)--(c) (b)--(d) (c)--(d);
 \ifdim\shift pt=0pt
   \draw[oiblue,line width=1.6pt] (a)--(b) (c)--(d);
 \else
   \draw[oiverm,line width=1.6pt] (a)--(b) (a)--(c);
   \draw[oiverm,thick,dashed] (a) circle[radius=.38cm];
 \fi
 \foreach \pos/\lab in {a/1,b/2,c/3,d/4}{\node[v] at (\pos) {\lab};}
 \node at (1.5,2.7) {(\letter)};
 \end{scope}
}
\node[align=center] at (1.5,-.8) {$\{12,34\}$ is a matching\\each vertex is used once};
\node[align=center] at (8.7,-.8) {$\{12,13\}$ repeats vertex $1$\\the selection is not a matching};
\end{tikzpicture}
\caption{Selecting edges within the complete graph on four vertices. In (a), the two thick edges form a matching, although the four thin edges are also present. In (b), the thick edges share vertex $1$. Here $ab$ denotes the edge $\{a,b\}$. The vertex positions serve only to show which edges share an endpoint; a linear order will be specified separately.}\label{fig:matching-selection}
\end{figure}

An \emph{ordered graph} also has a fixed linear order on its vertices. We label the vertices $1,\ldots,n$ in that order, and write $ab$ for the edge $\edge ab$ when $a<b$. For example, $13$ means $\{1,3\}$; we use braces in $\{8,10\}$ and $\{1,k+1\}$ to keep the endpoints unambiguous. The \emph{length} of an edge $ab$ is $b-a$, its distance in this vertex order. All graphs are simple, with at most one edge between any two distinct vertices. The \emph{complete ordered graph} $K_n$ contains every pair of vertices as an edge and hence has $\binom n2$ edges.

Suppose that two disjoint edges are $ab$ and $cd$, with $a<c$. They are \emph{separated} when $a<b<c<d$, \emph{crossing} when $a<c<b<d$, and \emph{nested} when $a<c<d<b$. Thus $12,34$ are separated, $13,24$ cross, and $14,23$ are nested. A matching is \emph{nonnested} if it contains no nested pair. Fig.~\ref{fig:three-positions} illustrates these three endpoint orders by placing the vertices along a line and drawing the selected edges as curves above it. The criterion is determined entirely by the endpoint order. In particular, the crossing in Fig.~\ref{fig:three-positions}(b) is an intersection of the drawn curves, and the edges still have four distinct endpoints.

\begin{figure}[H]
\centering
\begin{tikzpicture}[x=.9cm,y=1cm,font=\small,
 v/.style={circle,fill=black,inner sep=1.5pt}]
\foreach \shift/\letter/\title in {0/a/separated,5.4/b/crossing,10.8/c/nested}{
 \begin{scope}[xshift=\shift cm]
 \draw[gray!45] (0,0)--(3,0);
 \foreach \j in {1,...,4}{\node[v] at (\j-1,0) {};\node[below=4pt] at (\j-1,0) {\j};}
 \node at (1.5,1.55) {(\letter) \title};
 \end{scope}
}
\draw[oiblue,thick] (0,0) to[bend left=55] (1,0);
\draw[oiblue,thick] (2,0) to[bend left=55] (3,0);
\begin{scope}[xshift=5.4cm]
\draw[oiblue,thick] (0,0) to[bend left=55] (2,0);
\draw[oiblue,thick] (1,0) to[bend left=55] (3,0);
\end{scope}
\begin{scope}[xshift=10.8cm]
\draw[oiverm,thick,dashed] (0,0) to[bend left=55] (3,0);
\draw[oiverm,thick,dashed] (1,0) to[bend left=55] (2,0);
\end{scope}
\node at (1.5,-.8) {allowed};
\node at (7.5,-.8) {allowed};
\node at (13.5,-.8) {excluded};
\end{tikzpicture}
\caption{The three relative positions of two disjoint edges. In (a), the two endpoint intervals are disjoint. In (b), their endpoints alternate. In (c), the interval from $1$ to $4$ contains both endpoints of $23$. Only the third arrangement is forbidden in a nonnested matching.}\label{fig:three-positions}
\end{figure}

A \emph{$k$-matching} has exactly $k$ edges and therefore uses $2k$ distinct vertices. It is \emph{perfect in $G$} if it uses every vertex of the graph $G$ from which its edges were selected. On $2k$ vertices, every $k$-matching is perfect; on more than $2k$ vertices, some vertices may be left unused. For example, Fig.~\ref{fig:six-matchings} gives the complete list of nonnested perfect matchings on the six ordered vertices $1,\ldots,6$:
\[
\begin{aligned}
 M_1&=\{12,34,56\}, & M_2&=\{12,35,46\}, & M_3&=\{13,24,56\},\\
 M_4&=\{13,25,46\}, & M_5&=\{14,25,36\}. &&
\end{aligned}
\]
To obtain the list, consider the vertex $j$ matched to vertex $1$, meaning that the edge $1j$ belongs to the matching. The possible values are $j=2,3,4$. Choosing $j=5$ or $j=6$ leaves too many vertices between $1$ and $j$ to pair them all with vertices outside that interval, so another selected edge would be nested inside $1j$. The choices $j=2$ and $j=3$ give $M_1,M_2$ and $M_3,M_4$, respectively; $j=4$ gives only $M_5$.

\begin{figure}[H]
\centering
\begin{tikzpicture}[x=.65cm,y=1cm,font=\small,
 v/.style={circle,fill=black,inner sep=1.35pt}]
\foreach \shift/\yy/\name in {0/0/M_1,5.1/0/M_2,10.2/0/M_3,2.55/-2.5/M_4,7.65/-2.5/M_5}{
 \begin{scope}[xshift=\shift cm,yshift=\yy cm]
 \draw[gray!45] (0,0)--(5,0);
 \foreach \j in {1,...,6}{\node[v] at (\j-1,0) {};\node[below=3pt,font=\scriptsize] at (\j-1,0) {\j};}
 \node at (2.5,1.2) {$\name$};
 \end{scope}
}
\foreach \u/\v in {1/2,3/4,5/6}{\draw[oiblue,thick] (\u-1,0) to[bend left=55] (\v-1,0);}
\begin{scope}[xshift=5.1cm]
 \foreach \u/\v in {1/2,3/5,4/6}{\draw[oiblue,thick] (\u-1,0) to[bend left=55] (\v-1,0);}
\end{scope}
\begin{scope}[xshift=10.2cm]
 \foreach \u/\v in {1/3,2/4,5/6}{\draw[oiblue,thick] (\u-1,0) to[bend left=55] (\v-1,0);}
\end{scope}
\begin{scope}[xshift=2.55cm,yshift=-2.5cm]
 \foreach \u/\v in {1/3,2/5,4/6}{\draw[oiblue,thick] (\u-1,0) to[bend left=55] (\v-1,0);}
\end{scope}
\begin{scope}[xshift=7.65cm,yshift=-2.5cm]
 \foreach \u/\v in {1/4,2/5,3/6}{\draw[oiblue,thick] (\u-1,0) to[bend left=55] (\v-1,0);}
\end{scope}
\end{tikzpicture}
\caption{All five nonnested perfect matchings on six ordered vertices. Each drawing uses every vertex once. Separated pairs and crossing pairs can occur together, as in $M_2$, $M_3$ and $M_4$; in $M_5$, every pair crosses. The list supplies a concrete target family for the deletion question.}\label{fig:six-matchings}
\end{figure}

We denote the family of nonnested perfect matchings on $\intset{1}{2k}$ by $\mathcal N_k$. A graph on a larger ordered vertex set contains a member of this family when some $2k$ of its vertices, read from left to right, carry one of these edge patterns. The graph may contain additional edges between the selected vertices. For instance, replacing $1,\ldots,6$ by $2,4,5,7,8,9$ in $M_4$ gives the nonnested matching $\{25,48,79\}$ in an ordered graph on nine vertices. We write $\nu_{\mathrm{nn}}(G)$ for the largest number of edges in a nonnested matching of $G$.

The endpoint description makes the restriction useful in proofs. List the edges of a matching as $x_1y_1,\ldots,x_ky_k$, with $x_i<y_i$ and $x_1<\cdots<x_k$. This matching is nonnested exactly when $y_1<\cdots<y_k$: if two right endpoints occur in the opposite order, the corresponding edges are nested. Conversely, increasing right endpoints allow only separated or crossing pairs. This simple test underlies the constructions below. It will let us detect a forbidden nesting by comparing the ranks of the left and right endpoints.

\subsection{Deleting edges: which deletions destroy every matching?}\label{sec:intro-blockers}
Start with $K_n$ and delete edges until no member of $\mathcal N_k$ survives. The task is to do this with as few deletions as possible, or equivalently to retain as many edges as possible. A set $D$ of deleted edges is called a \emph{blocker} if every such matching contains at least one edge of $D$.

For $k=3$ and $n=6$, the set $D=\{12,13,14\}$ meets all five matchings in Fig.~\ref{fig:six-matchings}. The edge $12$ meets $M_1,M_2$, the edge $13$ meets $M_3,M_4$, and the edge $14$ meets $M_5$. Deleting these three edges therefore leaves $15-3=12$ edges and no nonnested perfect matching. Three deletions are necessary because \(M_2,M_3,M_5\) share no edges with each other: one deleted edge can meet at most one of these three matchings. Our classification describes every three-edge choice that works. In general, Theorem~\ref{blk:thm:main} shows that on $2k$ vertices a minimum blocker has $k$ edges and that there are exactly
\[
2^k+k-2
\]
minimum blockers for each $k\ge2$. Section~\ref{blk:sec:perfect-statement} gives the three explicit families and their proofs. The proof represents matchings by paths built from consecutive intervals. For any blocker, this representation identifies $k$ disjoint groups of edges that the blocker must meet. A blocker with only $k$ edges must place exactly one deletion in each group. Further matching tests determine which positions of these deletions are possible.

Allowing unused vertices changes the deletion problem. Define $f_k(n)$ to be the maximum number of edges in an ordered graph on $\intset{1}{n}$ containing no nonnested $k$-matching. The classification gives
$f_k(2k)=\binom{2k}{2}-k=2k(k-1)$.
For $n>2k$, a matching may avoid some vertices affected by the deletions, so the classification on $2k$ vertices does not determine $f_k(n)$. In their preprint, Barát, Freschi and Tóth~\cite[Theorem~2.5 and Conjecture~2.6]{BFT2025} proposed that
$f_k(n)=(k-1)n$ for every $n\ge2k$. They give the bounds
\begin{equation}\label{eq:intro-extremal-bounds}
(k-1)n\le f_k(n)\le (k-1)n+\binom{k-1}{2}.
\end{equation}
The lower bound says that $(k-1)n$ edges can be retained; the upper bound limits how many additional edges could be retained.

The reference value has a direct construction. Keep every edge among the first \(2k-1\) vertices, and join each later vertex only to the first \(k-1\) vertices. The number of retained edges is
\[
\binom{2k-1}{2}+(k-1)(n-2k+1)=(k-1)n.
\]
A \(k\)-matching cannot lie entirely on the first \(2k-1\) vertices. If a nonnested \(k\)-matching used a later vertex, its largest right endpoint \(y_k\) would be matched to a vertex \(x_k\le k-1\). All \(k\) left endpoints would then lie in \(\intset{1}{k-1}\), because \(x_1<\cdots<x_k\), which is impossible. Fig.~\ref{fig:intro-baseline} shows this construction on seven vertices.

\begin{figure}[htbp]
\centering
\begin{adjustbox}{max width=\linewidth}
\begin{tikzpicture}[x=0.95cm,font=\small,
  lab/.style={font=\scriptsize,inner sep=1pt},
  vtx/.style={circle,fill=black,inner sep=1.3pt},
  cl/.style={oiblue,semithick},
  jn/.style={oigreen,thick,dash pattern=on 4pt off 2pt}]
\fill[oiblue!7] (0.7,-0.25) rectangle (5.3,1.95);
\foreach \i in {1,...,5}{\foreach \j in {1,...,5}{
  \ifnum\j>\i
    \pgfmathsetmacro{\bend}{min(75,30+12*(\j-\i))}
    \draw[cl] (\i,0) to[bend left=\bend] (\j,0);
  \fi}}
\foreach \u/\v/\bend in {1/6/30,2/6/22,1/7/40,2/7/33}{
  \draw[jn] (\u,0) to[bend right=\bend] (\v,0);}
\draw[gray!45] (1,0)--(7,0);
\foreach \v in {1,...,7}{\node[vtx] at (\v,0) {};\node[lab,inner sep=0.6pt] at (\v,0.27) {\v};}
\draw[decorate,decoration={brace,amplitude=4pt}] (0.8,2.15)--(5.2,2.15) node[midway,above=4pt,lab] {complete graph on $2k-1=5$ vertices: $\binom52=10$ edges};
\node[lab,anchor=west,align=left] at (7.35,0.9) {$6$ and $7$ are joined\\only to vertices $1,2$:\\$2+2=4$ edges};
\node[lab,anchor=west,align=left] at (7.35,-0.75) {total $10+4=14=(k-1)n$};
\end{tikzpicture}
\end{adjustbox}
\caption{The graph giving the lower bound \((k-1)n\), for \(k=3\) and \(n=7\). The solid edges form the complete graph on vertices \(1,\ldots,5\); the dashed edges join \(6\) and \(7\) only to \(1,2\). A nonnested triple using \(6\) or \(7\) would need three distinct left endpoints among just \(1,2\).}\label{fig:intro-baseline}
\end{figure}

The proposed equality fails for every $k\ge5$ and $n\ge2k+1$: Section~\ref{blk:sec:beyond} constructs graphs with $(k-1)n+1$ edges and no target matching. We then coordinate the deletions near both ends of the order to retain still more edges. Corollary~\ref{ext:cor:coordinated-envelope} gives
\begin{equation}\label{eq:intro-quadratic}
f_k(n)\ge (k-1)n+\binom K2,
\qquad
K=\left\lfloor\frac{k+1}{3}\right\rfloor,\quad
k\ge5,\quad n\ge2k+2K-3.
\end{equation}
Here the \emph{excess} means precisely $f_k(n)-(k-1)n$, the number of extra edges above the proposed value. Since $K$ is about $k/3$, the additional term $\binom K2$ is about $k^2/18$. The upper bound on this excess is about $k^2/2$. Thus the excess grows quadratically in $k$ throughout the range of Eq.~\eqref{eq:intro-quadratic}; the simpler condition $n\ge3k$ suffices for that range. The exact value of $f_k(n)$ and the best leading constant remain open in general. We include a proof of the upper bound from the published theorem of Capoyleas and Pach~\cite{CP1992} in Proposition~\ref{ext:prop:ext-upper-bounds}.

\subsection{Proof structure and notation}\label{sec:intro-reading}
The perfect and nonperfect cases require different information about a matching. In the perfect case, every vertex must be used, and the interval-cut argument determines every smallest obstruction. With unused vertices, we instead track the ranks of the surviving endpoints; these ranks explain how coordinated deletions at the two ends can block every matching while retaining extra edges.

Section~\ref{blk:sec:perfect-statement} gives the minimum-blocker classification and its proof. Section~\ref{blk:sec:consequences} turns the proof into explicit certificates and relates the blockers to a matrix problem. Section~\ref{blk:sec:beyond} constructs graphs on more than $2k$ vertices, optimizes the excess within the construction, and compares it with the general upper bound. The first version of this paper contained the classification and the construction with one extra edge. The present version adds the coordinated endpoint constructions and the quadratic excess.

We write $V(G)$ and $E(G)$ for the vertices and edges of $G$, and $e(G)=|E(G)|$ for its edge count. The degree of a vertex is the number of edges containing it. For integers $a$ and $b$, we write $\intset{a}{b}$ for the set of all integers from $a$ through $b$, including both endpoints. This set is empty when $a>b$. An interval of vertices is \emph{initial} or a \emph{prefix} when it starts at vertex $1$. Specialized constructions will be defined where they enter the proofs.
The notation below collects the objects already introduced.
\begin{center}
\small
\begin{tabular}{@{}lp{9.3cm}@{}}
\toprule
Notation & Meaning\\
\midrule
$\mathcal N_k$ & Nonnested perfect-matching patterns on $2k$ ordered vertices\\
$\nu_{\mathrm{nn}}(G)$ & Largest number of edges in a nonnested matching of $G$\\
$f_k(n)$ & Most edges that can remain on $\intset{1}{n}$ while avoiding $\mathcal N_k$\\
\bottomrule
\end{tabular}
\end{center}

\section{Minimum blockers for perfect matchings}\label{blk:sec:perfect-statement}
On \(2k\) vertices, a matching with \(k\) edges uses every vertex. We ask how few edges can be deleted from the complete ordered graph to destroy all nonnested perfect matchings, and then determine every deletion set attaining that minimum. The proof starts with matchings assembled from consecutive blocks. Their simple structure gives a lower bound of \(k\) deletions and restricts the equality case enough to classify the full family.

Throughout this section, \(\mathcal N_k\) denotes the nonnested perfect matchings on the fixed set \(\intset{1}{2k}\). A blocker meets each of these matchings in at least one edge, and a \emph{minimum} blocker has the smallest possible number of edges among all blockers. An \emph{inclusion-minimal} blocker loses the blocking property when any one of its edges is removed. Our classification concerns the smallest cardinality.

The choice of permitted endpoint orders matters to this problem. A matching is \emph{noncrossing} when it has no two edges \(ab,cd\) with \(a<c<b<d\), the crossing arrangement in Fig.~\ref{fig:three-positions}(b). Nonnested and noncrossing perfect matchings are classical Catalan families~\cite[Section~5]{Klazar2000}, meaning that both families have \(\frac{1}{k+1}\binom{2k}{k}\) members on \(2k\) ordered vertices. For noncrossing matchings drawn as chords of a convex polygon, Keller and Perles classified the minimum blockers and obtained \(k2^{k-1}\) choices~\cite{KP2012}; their later work classifies minimum sets meeting those blockers~\cite{KP2013}. Allowing crossings and excluding nesting gives the different families and count below. Subsection~\ref{blk:sec:matrix} and Proposition~\ref{blk:prop:coblocker} will further compare the associated blocking problems.

We first describe the families that will remain after the proof. A \emph{fan} here is a set of edges having one endpoint in common. For example, on ten vertices the left blocker \(\{12,13,25,26,38\}\) starts with the fan \(\{12,13\}\) at vertex \(1\). Its remaining edges have lengths \(3,4,5\), and their smaller endpoints \(2,2,3\) increase weakly. This gives a convenient way to specify one deleted edge of every length from \(1\) to \(k\).

The same construction can be placed at the right end by reversing the vertex order. Formally, let \(\rho(v)=2k+1-v\) and, for an edge set \(D\), let
\(\rho(D)=\{\rho(b)\rho(a):ab\in D,\ a<b\}\).
The following formulas describe the left family and a third family using fans at both ends.
For \(1\le m\le k\), take a weakly increasing sequence
\[
2\le a_{m+1}\le\cdots\le a_k\le m+1,
\]
allowing the empty sequence when \(m=k\), and put
\begin{equation}\label{blk:eq:left}
L(m,a)=\{1i:2\le i\le m+1\}
 \cup\{\edge{a_j}{a_j+j}:m+1\le j\le k\}.
\end{equation}
For \(1\le m\le k-2\), put
\begin{equation}\label{blk:eq:mixed}
C_m=\{1i:2\le i\le m+1\}
 \cup\{\edge{m+1}{k+m+1}\}
 \cup\{\edge{j}{2k}:k+m+1\le j\le2k-1\}.
\end{equation}
The set \(L(m,a)\) contains \(m\) edges in the fan at \(1\), followed by one edge of each length \(m+1,\ldots,k\). The set \(C_m\) contains \(m\) edges at \(1\), one edge of length \(k\), and \(k-m-1\) edges at \(2k\), for a total of \(k\). There are no sets \(C_m\) when \(k=2\). Fig.~\ref{blk:fig:families} shows all three forms for the same ten vertices.

\begin{theorem}\label{blk:thm:main}
For \(k\ge2\), the minimum blockers of \(\NN\) have \(k\) edges and
are exactly
\[
 L(m,a),\qquad \rho(L(m,a)),\qquad C_m.
\]
These classes are disjoint and contain \(2^{k-1}\), \(2^{k-1}\) and \(k-2\)
members, respectively. Consequently their total number is
\begin{equation}\label{blk:eq:count}
b_k=2^k+k-2.
\end{equation}
Equivalently, the largest ordered graphs on \(\intset{1}{2k}\) without a nonnested
perfect matching have \(2k(k-1)\) edges, and are the complements of the
displayed blockers.
\end{theorem}

Subsection~\ref{blk:sec:cuts} proves that every blocker has at least \(k\) edges and identifies all equality forms for the simpler block family. Subsection~\ref{blk:sec:classification} then tests those forms against the remaining nonnested matchings. Explicit avoiding matchings prove necessity, and a monotone integer iteration proves sufficiency for the left family.

The theorem says both how many deletions are necessary and where every optimal set of deletions can lie. For \(k=3\), it gives four left blockers, their four reflections, and the mixed blocker \(C_1=\{12,25,56\}\), hence nine blockers altogether. The five matchings in Fig.~\ref{fig:six-matchings} illustrate the mixed blocker directly: \(M_1,M_2\) meet it at \(12\), \(M_3\) meets it at \(56\), and \(M_4,M_5\) meet it at \(25\). Thus the same six-vertex family is blocked by a full fan in the introduction and by two end edges together with \(25\) here. The complete list for \(k=3\) is small enough to display:
\begin{center}
\small
\begin{tabular}{@{}lll@{}}
\toprule
Left blocker & Its reflection & Mixed blocker\\
\midrule
\(\{12,13,14\}\) & \(\{36,46,56\}\) & \(\{12,25,56\}\)\\
\(\{12,13,25\}\) & \(\{25,46,56\}\) & \\
\(\{12,13,36\}\) & \(\{14,46,56\}\) & \\
\(\{12,24,25\}\) & \(\{25,35,56\}\) & \\
\bottomrule
\end{tabular}
\end{center}
Each listed set meets all five matchings in Fig.~\ref{fig:six-matchings}; the theorem asserts that no other three-edge set does so.

\begin{figure}[htbp]
\centering
\begin{tikzpicture}[x=0.80cm,y=0.62cm,
  vertex/.style={circle,fill=black,inner sep=1.3pt},
  fan/.style={line width=0.75pt},
  branch/.style={line width=0.85pt,dashed}]
\foreach \yy/\name in {0/{(a) \(L(2,a)\)},-3.35/{(b) \(\rho(L(2,a))\)},-6.7/{(c) \(C_2\)}} {
  \node[anchor=east,font=\small] at (0.35,\yy+0.55) {\name};
  \draw[gray!45,line width=0.35pt] (1,\yy)--(10,\yy);
  \foreach \v in {1,...,10} {
    \node[vertex] at (\v,\yy) {};
    \node[below=4pt,font=\small] at (\v,\yy) {\v};
  }
}
\foreach \a/\b in {1/2,1/3} {
 \draw[fan] (\a,0) .. controls (\a,{0.45*(\b-\a)}) and (\b,{0.45*(\b-\a)}) .. (\b,0);
}
\foreach \a/\b in {2/5,2/6,3/8} {
 \draw[branch] (\a,0) .. controls (\a,{0.45*(\b-\a)}) and (\b,{0.45*(\b-\a)}) .. (\b,0);
}
\foreach \a/\b in {9/10,8/10} {
 \draw[fan] (\a,-3.35) .. controls (\a,{-3.35+0.45*(\b-\a)}) and (\b,{-3.35+0.45*(\b-\a)}) .. (\b,-3.35);
}
\foreach \a/\b in {3/8,5/9,6/9} {
 \draw[branch] (\a,-3.35) .. controls (\a,{-3.35+0.45*(\b-\a)}) and (\b,{-3.35+0.45*(\b-\a)}) .. (\b,-3.35);
}
\foreach \a/\b in {1/2,1/3,8/10,9/10} {
 \draw[fan] (\a,-6.7) .. controls (\a,{-6.7+0.45*(\b-\a)}) and (\b,{-6.7+0.45*(\b-\a)}) .. (\b,-6.7);
}
\draw[branch] (3,-6.7) .. controls (3,-4.45) and (8,-4.45) .. (8,-6.7);
\end{tikzpicture}
\caption{The three types on ten vertices. In (a),
\(L(2,a)=\{12,13,25,26,38\}\), with \(a=(2,2,3)\).
In (b), reflection gives
\(\{38,59,69,\edge{8}{10},\edge{9}{10}\}\).
In (c), \(C_2=\{12,13,38,\edge{8}{10},\edge{9}{10}\}\)
uses fans at both ends and the edge \(38\) between them.
Solid edges form the fans, and the remaining edges are dashed.
Every row contains five deletions, including one of length five.}\label{blk:fig:families}
\end{figure}

In particular, \(b_k=\mathrm{A083706}(k-1)\) in the OEIS~\cite{OEIS};
the classification gives this sequence
a blocker interpretation. For \(k=1\), the sole minimum blocker is
\(\{12\}\), so \(b_1=1\). Unless stated otherwise, assume \(k\ge2\).
At the level of individual examples, the first construction in
\cite[Section~2.2]{BFT2025}, applied to \(K_{2k-1}\) and reflected,
gives the complement of the full fan \(L(k,\varnothing)\).
Its second construction gives the complement of \(C_1\) when \(k=3\)
and supplies the remaining examples when \(k=2\).
The classification above establishes exhaustiveness for every \(k\).

\subsection{Interval cuts and the equality case}\label{blk:sec:cuts}
A blocker must meet even a restricted family of nonnested matchings. We choose matchings that can be built by dividing the vertex order into consecutive intervals and matching across the middle of each interval. An auxiliary path records this division. If deletions destroy every such path, the edges crossing from reachable to unreachable indices give enough disjoint matching blocks to count the necessary deletions.

For \(0\le i<j\le k\), define
\[
 B_{ij}=\{\edge{a}{a+j-i}:2i+1\le a\le i+j\}.
\]
For example, \(B_{02}=\{13,24\}\): the index pair \(0,2\) represents a block with two original edges. The interval \(\intset{2i+1}{2j}\) has \(2(j-i)\) vertices. Its first \(j-i\) vertices are paired, in increasing order, with its last \(j-i\) vertices, so every two edges of \(B_{ij}\) cross. We call this the fully crossing perfect matching on that interval.

An increasing path \(0=i_0<i_1<\cdots<i_s=k\) now produces the matching
\(B_{i_0i_1}\cup\cdots\cup B_{i_{s-1}i_s}\).
Its blocks occupy consecutive intervals, so edges from different blocks are separated. Write \(\mathcal P_k\) for the family of all matchings obtained this way; it is a subfamily of \(\mathcal N_k\). For example, when \(k=5\), the path \(0,2,3,5\) gives
\[
 B_{02}\cup B_{23}\cup B_{35}
 =\{13,24\}\cup\{56\}\cup\{79,\edge{8}{10}\}.
\]
On the six vertices used in Fig.~\ref{fig:six-matchings}, the four possible paths are \(0,1,2,3\), \(0,1,3\), \(0,2,3\), and \(0,3\). They yield \(M_1,M_2,M_3,M_5\), respectively. Thus \(\mathcal P_3\) omits just \(M_4=\{13,25,46\}\), making the distinction between the restricted family and the full family explicit. In general, \(\mathcal P_k\) is precisely the perfect-matching subfamily of the strongly nonnested family of~\cite[Section~2.2]{BFT2025}.

For a set \(D\) of deleted edges, form an auxiliary directed graph on \(\{0,\ldots,k\}\). Retain the arc \(i\longrightarrow j\) exactly when its whole block \(B_{ij}\) avoids \(D\). A single deletion can therefore destroy a block: when \(k=3\) and \(D=\{13\}\), the auxiliary arc \(0\longrightarrow2\) disappears because \(13\in B_{02}\). This deletion still leaves \(M_2=\{12,35,46\}\), so it does not yet block the entire family.

Now assume that \(D\) meets every member of \(\mathcal P_k\). A surviving path from \(0\) to \(k\) would give a matching avoiding \(D\), so no such path exists. If \(S\) is the set of indices reachable from \(0\), then \(0\in S\) and \(k\notin S\). Every forward arc leaving \(S\) must therefore have a deleted edge in its block.

We record these destroyed arcs in an undirected graph \(Q\), with
\[
 E(Q)=\{ij:i<j,\ i\in S,\ j\notin S\},
\]
and give \(ij\) weight \(j-i\), the common length of the original edges in \(B_{ij}\). The graph \(Q\) is connected: \(0\) is adjacent to every index outside \(S\), and \(k\) to every index in \(S\). Choose a spanning tree of \(Q\), that is, a connected acyclic subgraph containing all \(k+1\) indices, whose total weight is smallest. Such a tree has \(k\) edges. The next lemma shows why minimizing the weight is useful: its \(k\) blocks have no original edge in common, so each requires a different deletion. The blocks may share vertices.

\begin{lemma}\label{blk:lem:tree}
If \(T\) is a minimum spanning tree of \(Q\), its \(k\) associated
matchings \(B_{ij}\), \(ij\in E(T)\), are pairwise edge-disjoint.
Consequently \(|D|\ge k\).
\end{lemma}
\begin{proof}
A block of weight \(d\) consists entirely of original edges of length \(d\), so blocks of different weights are disjoint. Suppose instead that two tree edges \(ij\) and \(i'j'\) have the same weight \(d\), with \(i<i'\), and their blocks share an edge. The smaller endpoints available in these blocks form the intervals \(\intset{2i+1}{2i+d}\) and \(\intset{2i'+1}{2i'+d}\). If \(h=i'-i\), their overlap implies \(2h<d\), so there is an integer \(v\) with \(i'<v<j\).

We use \(v\) to replace one weight-\(d\) tree edge by a lighter edge. If \(v\in S\), the path \(i',j,v,j'\) lies in \(Q\), joins the endpoints of \(i'j'\), and has every weight smaller than \(d\). If \(v\notin S\), the path \(i,v,i',j\) has the same properties for \(ij\). In either case, deleting the indicated edge from the tree separates it into two components. The lighter path must cross between these components, so one of its edges reconnects the tree with a smaller total weight. This contradicts the choice of the tree.

The \(k\) tree blocks are therefore pairwise edge-disjoint. Each block meets \(D\), because its associated arc leaves \(S\). They require at least \(k\) distinct deleted edges, proving \(|D|\ge k\).
\end{proof}

The lower bound is attained by the full fan
\(\{12,13,\ldots,\edge{1}{k+1}\}\).
Indeed, vertex \(1\) is the first left endpoint of every perfect matching, and the first of its \(k\) right endpoints is at most \(k+1\). Their edge belongs to this fan. Thus the minimum blocker size is \(k\) for both \(\mathcal N_k\) and \(\mathcal P_k\), and equality forces every tree block to contain exactly one deletion.

We next determine what equality permits. The term \emph{split} will refer to an index \(t\) separating the unreachable interior indices \(1,\ldots,t-1\) from the reachable interior indices \(t,\ldots,k-1\). The lemma proves that every blocker of size at most \(k\) has this form in the auxiliary graph.

\begin{lemma}\label{blk:lem:equality}
Suppose \(|D|\le k\) and the surviving auxiliary graph has no path
from \(0\) to \(k\). Then \(|D|=k\), and for some \(1\le t\le k\),
\begin{equation}\label{blk:eq:cutform}
D=
 \{\edge{a_d}{a_d+d}:1\le d<t\}
 \cup\{\edge{c}{c+k}\}
 \cup\rho\bigl(\{\edge{b_d}{b_d+d}:1\le d\le k-t\}\bigr),
\end{equation}
where \(1\le a_d,b_d\le d\) and \(1\le c\le k\).
Moreover the reachable set is
\(S=\{0\}\cup\{t,t+1,\ldots,k-1\}\).
Conversely, every choice in Eq.~\eqref{blk:eq:cutform} is a minimum blocker of
\(\PP\).
\end{lemma}
The edge \(\edge{c}{c+k}\) in this description will be called the central deletion. All other deletions come from a block at the left end or from the reflection of such a block.

\begin{proof}
Lemma~\ref{blk:lem:tree} gives \(|D|=k\). Its tree has \(k\) disjoint blocks, each meeting \(D\), so every deletion belongs to exactly one tree block. The direct arc \(0\longrightarrow k\) must be destroyed, hence \(D\) contains an edge of \(B_{0k}\), of length \(k\). The only auxiliary arc of weight \(k\) is \(0k\), so this edge belongs to the tree.

Every other edge of \(Q\) has weight below \(k\). Consequently \(Q-0k\) has no path from \(0\) to \(k\): such a path would cross the two components of the tree after removing \(0k\), allowing a lighter replacement. In particular, no two interior indices \(i<j\) can satisfy \(i\in S\) and \(j\notin S\), because \(0,j,i,k\) would then be a path in \(Q-0k\). Once an interior index belongs to \(S\), every later interior index must belong to \(S\) as well. This proves
\(S=\{0\}\cup\{t,\ldots,k-1\}\) for some \(1\le t\le k\).

For this set \(S\), the entire graph \(Q\) is already a tree. Its edges are \(0d\) for \(1\le d<t\), the edge \(0k\), and \(ik\) for \(t\le i<k\). Choosing one deletion from each corresponding block gives precisely Eq.~\eqref{blk:eq:cutform}: the blocks at \(0\) give the \(a_d\), the block \(B_{0k}\) gives \(c\), and the blocks at \(k\) give the reflected \(b_d\).

Conversely, each choice in Eq.~\eqref{blk:eq:cutform} destroys every forward arc leaving the displayed set \(S\). Every path from \(0\) to \(k\) must leave \(S\), so every member of \(\mathcal P_k\) meets the deletion set. The blocks just listed are edge-disjoint, giving exactly \(k\) deletions.
\end{proof}

Fig.~\ref{blk:fig:interval-cut} follows this equality case for $k=5$ and $D=C_2$. The five cut edges on the auxiliary levels correspond to five edge-disjoint blocks on the original ten vertices, each containing one deletion.

\begin{figure}[htbp]
\centering
\begin{adjustbox}{max width=\linewidth}
\begin{tikzpicture}[font=\small,
  lab/.style={font=\small,inner sep=1pt},
vlab/.style={font=\small,inner sep=1pt},
  snode/.style={circle,draw=black,fill=oiorange!30!white,text=black,inner sep=0pt,minimum size=5.2mm,font=\small},
  onode/.style={circle,draw=black,fill=white,text=black,inner sep=0pt,minimum size=5.2mm,font=\small},
  alive/.style={-{Stealth[length=1.7mm]},black,line width=.55pt,shorten <=2.6mm,shorten >=2.6mm},
  dead/.style={-{Stealth[length=1.8mm]},oiverm,line width=.8pt,dashed,shorten <=2.6mm,shorten >=2.6mm},
  vtx/.style={circle,fill=black,inner sep=1.1pt},
  blk/.style={black,line width=.45pt},
  del/.style={oiverm,line width=1.1pt}]

\node[anchor=west] at (-0.55,2.55) {(a) Auxiliary directed graph};
\foreach \i in {0,...,5}{\coordinate (a\i) at (2.55*\i,0);}
\foreach \i/\j in {1/2,2/3,3/4}{\draw[alive] (a\i)--(a\j);}
\foreach \i/\j/\h in {0/3/1.65,0/4/2.70,1/3/.85,1/4/1.65,1/5/2.70,2/4/.85,2/5/1.65}{
  \pgfmathsetmacro{\xa}{2.55*\i+.65}
  \pgfmathsetmacro{\xb}{2.55*\j-.65}
  \draw[alive] (a\i) .. controls (\xa,\h) and (\xb,\h) .. (a\j);}
\foreach \i/\j/\h/\deletion in {0/1/.70/{12},0/2/1.35/{13},0/5/2.45/{38},3/5/1.35/{\{8,10\}},4/5/.70/{\{9,10\}}}{
  \pgfmathsetmacro{\xa}{2.55*(\i+.28*(\j-\i))}
  \pgfmathsetmacro{\xb}{2.55*(\j-.28*(\j-\i))}
  \draw[dead] (a\i) .. controls (\xa,-\h) and (\xb,-\h) ..
    node[lab,pos=.5,above=3pt,text=black] {$\deletion$} (a\j);}
\foreach \i in {0,3,4}{\node[snode] at (a\i) {\i};}
\foreach \i in {1,2,5}{\node[onode] at (a\i) {\i};}

\begin{scope}[yshift=-3.45cm]
\node[anchor=west] at (-0.55,0.7) {(b) Cut tree $Q=T$};
\node[snode] (t0) at (0.6,-1.0) {0};
\node[onode] (t5) at (2.8,-1.0) {5};
\node[onode] (t1) at (-0.05,0.05) {1};
\node[onode] (t2) at (-0.05,-2.05) {2};
\node[snode] (t3) at (3.45,0.05) {3};
\node[snode] (t4) at (3.45,-2.05) {4};
\draw[oiverm,line width=.8pt] (t0)--node[lab,text=black,left=3pt] {$1$} (t1);
\draw[oiverm,line width=.8pt] (t0)--node[lab,text=black,left=3pt] {$2$} (t2);
\draw[oiverm,line width=.8pt] (t0)--node[lab,text=black,above=3pt] {$5$} (t5);
\draw[oiverm,line width=.8pt] (t3)--node[lab,text=black,right=3pt] {$2$} (t5);
\draw[oiverm,line width=.8pt] (t4)--node[lab,text=black,right=3pt] {$1$} (t5);
\node[lab,align=center] at (1.7,-2.8) {$S=\{0,3,4\}$\\[3pt]weights $j-i$};
\end{scope}

\begin{scope}[xshift=5.3cm,yshift=-3.45cm,x=.7cm]
\node[anchor=west] at (-.7,0.7) {(c) Edge-disjoint matching blocks};
\foreach \row/\name/\edges/\deleted in {
  0/{B_{01}}/{1/2}/{1/2},
  1/{B_{02}}/{1/3,2/4}/{1/3},
  2/{B_{05}}/{1/6,2/7,3/8,4/9,5/10}/{3/8},
  3/{B_{35}}/{7/9,8/10}/{8/10},
  4/{B_{45}}/{9/10}/{9/10}}{
  \begin{scope}[yshift=-\row*.91cm]
  \node[lab,anchor=east] at (.55,.05) {$\name$};
  \foreach \u/\w in \edges {\draw[blk] (\u,0) to[bend left=38] (\w,0);}
  \foreach \u/\w in \deleted {\draw[del] (\u,0) to[bend left=38] (\w,0);}
  \foreach \v in {1,...,10}{\node[vtx] at (\v,0) {};}
  \end{scope}}
\foreach \v in {1,...,10}{\node[vlab] at (\v,-3.9) {\v};}
\end{scope}
\end{tikzpicture}
\end{adjustbox}
\caption{The interval cut for $k=5$ and $D=C_2=\{12,13,38,\{8,10\},\{9,10\}\}$. In (a), black arcs survive and dashed orange arcs are destroyed; each destroyed arc is labelled by the original edge of $D$ in its block. The beige vertices form the reachable set $S=\{0,3,4\}$. In (b), the five arcs leaving $S$ form the tree $Q=T$ when directions are suppressed, with weights $j-i$. In (c), its five matching blocks share no original edge, and each contains exactly one deletion, shown in orange. This is the equality mechanism in Lemma~\ref{blk:lem:equality}.}\label{blk:fig:interval-cut}
\end{figure}

The end split \(t=k\) contains \(12\) and omits \(\edge{2k-1}{2k}\).
The split \(t=1\) has the reflected property. An internal split contains
both end edges. This distinction will separate the three classes.

\begin{remark}[The interval-block family]\label{blk:cor:concat}
The relaxed family \(\PP\) has more minimum blockers than \(\NN\) in general. Its count is
\[
 k\sum_{j=0}^{k-1}j!(k-1-j)!.
\]
\end{remark}
\begin{proof}
For a fixed split \(t\), the left choices number \((t-1)!\), because \(a_d\) has \(d\) possible values for each \(1\le d<t\). The reflected choices similarly number \((k-t)!\), and the length-\(k\) edge has \(k\) possible positions. Thus this split gives \(k(t-1)!(k-t)!\) different sets.

To add these counts, we must also distinguish sets obtained from different splits. The multiset of noncentral deletion lengths determines the unordered pair \(\{t-1,k-t\}\). Two different splits giving the same lengths would therefore be \(t<t'=k+1-t\), with \(d=t\le k/2\). Their unique deletion of length \(d\) comes from opposite ends: its smaller endpoint is at least \(2k+1-2d\) for the first split and at most \(d\) for the second. These ranges are disjoint, so the deletion sets differ. Summing the counts over \(t=1,\ldots,k\) proves the formula.
\end{proof}

\subsection{Classification of nonnested blockers}\label{blk:sec:classification}
A size-\(k\) blocker of the full family must be one of the equality forms in Lemma~\ref{blk:lem:equality}, because it also blocks \(\mathcal P_k\). Some of those forms still allow other nonnested perfect matchings. We eliminate them with explicit matchings, treating internal splits \(2\le t\le k-1\) and end splits \(t=1,k\) separately. We then prove directly that every surviving form is a blocker.

The endpoint order first gives a useful restriction on every candidate matching. If \(ab=x_i y_i\) has rank \(i\) in the increasing endpoint lists, its \(i-1\) earlier left endpoints lie before \(a\), and its \(k-i\) later right endpoints lie after \(b\). These \(k-1\) vertices fit into the available positions only if
\[
 k-1\le (a-1)+(2k-b).
\]
Rearranging gives \(b-a\le k\): every edge of a nonnested perfect matching on \(\intset{1}{2k}\) has length at most \(k\).

For \(1\le p\le k-1\), consider the following matching, designed to have one edge of length \(p\) and remaining edges of lengths \(k\) and \(k-1\):
\begin{equation}\label{blk:eq:witness}
P_p=\{\edge{1}{p+1}\}
 \cup\{\edge{i}{k+i}:2\le i\le p\}
 \cup\{\edge{i+1}{k+i}:p+1\le i\le k\}.
\end{equation}
Its three parts have \(1\), \(p-1\), and \(k-p\) edges, of lengths \(p\), \(k\), and \(k-1\), respectively. They use every vertex once, and their left and right endpoints increase together, so \(P_p\) is nonnested. For instance, on ten vertices \(P_3=\{14,27,38,59,\edge{6}{10}\}\). Comparing these three edge lengths with the deletion lengths in a split will leave very few possible intersections.

\subsubsection{Necessity: internal splits}
An internal split has deletions associated with both ends of the order and one edge of length \(k\). The matchings \(P_p\) first force the position of that long edge. Once its position is fixed, the same matchings force all the remaining deletions to form the two fans of \(C_{t-1}\).

For \(k=3\), the only internal split is \(t=2\). Eq.~\eqref{blk:eq:cutform} allows the three sets
\[
D_c=\{12,\edge c{c+3},56\},\qquad c=1,2,3.
\]
All three block \(\mathcal P_3\). The omitted matching \(M_4=\{13,25,46\}\) avoids \(D_1\) and \(D_3\), forcing \(c=2\) and leaving exactly \(C_1=\{12,25,56\}\). The following argument performs the same exclusion for every internal split.

Let \(D\) be a minimum blocker of \(\mathcal N_k\), with \(2\le t\le k-1\) in Eq.~\eqref{blk:eq:cutform}. Every noncentral deletion has length at most \(k-2\). The first edge \(\edge{1}{t+1}\) of \(P_t\) is also absent from \(D\): the left part contains only lengths below \(t\), and the reflected part has no edge at \(1\). Thus the only possible intersection is with the central deletion \(\edge{c}{c+k}\). The length-\(k\) edges of \(P_t\) start at \(2,\ldots,t\), forcing \(2\le c\le t\).

Reflection changes the parameters to \(k+1-t\) and \(k+1-c\). Applying the same argument to the reflected blocker gives \(k+1-c\le k+1-t\), hence \(c\ge t\). Therefore \(c=t\).

For every \(p<t\), the length-\(k\) edges of \(P_p\) start before \(t\), and its length-\((k-1)\) edges are longer than every noncentral deletion. Its only possible intersection with \(D\) is consequently its first edge \(\edge{1}{p+1}\), which must be deleted. These choices give the full left fan, and reflection gives the full right fan. Along with \(c=t\), they yield \(D=C_{t-1}\).

\subsubsection{Necessity: end splits}
For an end split, Eq.~\eqref{blk:eq:cutform} gives one deletion of every length from \(1\) to \(k\). We recover their permitted positions one length at a time. Removing the longest deletion leaves the same problem on the first \(2k-2\) vertices, and an explicit avoiding matching rules out each invalid way of restoring that deletion.

By reflection, it is enough to treat \(t=k\), for which
\[
 D=\{\edge{a_d}{a_d+d}:1\le d\le k\},\qquad 1\le a_d\le d.
\]
We use induction on \(k\), starting with the sole blocker \(\{12\}\) on two vertices. Put \(c=a_k\) and remove the length-\(k\) edge \(\edge{c}{c+k}\), leaving \(D'\). The remaining edges lie in \(\intset{1}{2k-2}\), since their larger endpoints satisfy \(a_d+d\le2d\le2k-2\).

If a nonnested perfect matching on this prefix avoided \(D'\), appending \(\edge{2k-1}{2k}\) would produce a matching avoiding \(D\). The appended edge is absent from an end split of this orientation. The prefix matching cannot contain the removed length-\(k\) edge either, by the length bound with \(k-1\) in place of \(k\). Hence \(D'\) is a minimum blocker on the prefix.

When \(k-1\ge2\), the set \(D'\) contains \(12\) and omits the other end edge of the prefix. The inductive classification therefore gives \(D'=L(m,a')\) for some \(1\le m\le k-1\); the one-edge base has the same form when \(k=2\). If \(m=k-1\), all values \(1\le c\le k\) extend it to a left blocker: \(c=1\) enlarges the full fan, while \(c\ge2\) starts a one-term sequence. The next lemma handles an existing nonempty sequence.

\begin{lemma}\label{blk:lem:extension}
Let \(k\ge3\), let \(D'=L(m,a')\) on \(\intset{1}{2k-2}\) with
\(1\le m\le k-2\), and set \(b=a'_{k-1}\).
If \(D'\cup\{\edge{c}{c+k}\}\) blocks \(\NN\), where \(1\le c\le k\),
then \(b\le c\le m+1\).
\end{lemma}
\begin{proof}
First suppose \(c=1\) or \(c>m+1\). Then \(P_{m+1}\) avoids the whole deletion set. Its first edge \(\edge{1}{m+2}\) lies beyond the old fan, and its length-\(k\) edges start at \(2,\ldots,m+1\), avoiding the new deletion. Its final edges have length \(k-1\) and start at \(m+3\) or later, beyond every smaller endpoint of an old nonfan deletion. Thus these values of \(c\) are impossible.

It remains to exclude \(2\le c<b\). Define complementary sets
\begin{align*}
Y&=\{m+2\}\cup\intset{k+1}{k+c-1}\cup\intset{k+c+1}{2k},\\
X&=\intset{1}{m+1}\cup\intset{m+3}{k}\cup\{k+c\}.
\end{align*}
Each has \(k\) elements, with an interval interpreted as empty when its lower endpoint exceeds its upper endpoint. Pair the increasing list of \(X\) with the increasing list of \(Y\). Because \(c<b\le m+1\), we have \(c\le m\le k-2\); every member of \(X\) is therefore smaller than the corresponding member of \(Y\), including the final pair \((k+c,2k)\). This constructs a nonnested perfect matching.

We check every kind of deletion against this matching. Its first edge is \(\edge{1}{m+2}\), beyond the fan. At smaller endpoints \(2,\ldots,c\), its edges have length \(k-1\); at smaller endpoints \(c+1,\ldots,m+1\), they have length \(k\). Thus every old nonfan deletion of length below \(k-1\) is avoided, and later left endpoints cannot meet an old deletion because they exceed \(m+1\). The old length-\((k-1)\) deletion starts at \(b>c\), whereas the new length-\(k\) deletion starts at \(c\). Both are avoided as well.

We have constructed an avoiding matching for every \(c\) outside \(\intset{b}{m+1}\). A blocker must consequently satisfy \(b\le c\le m+1\), as claimed.
\end{proof}

The condition \(b\le c\le m+1\) appends a weakly increasing entry in the permitted range of Eq.~\eqref{blk:eq:left}. Together with the full-fan case, this completes the induction and proves necessity for the left family. Reflection proves necessity for the right family.

\subsubsection{Sufficiency and counting}
We have reduced every minimum blocker to one of the stated forms. We must still check that each form meets every nonnested perfect matching. For a left blocker, avoiding the fan fixes the first left endpoints, and the increasing sequence of deleted positions then forces a common edge. The two-fan form admits a direct rank argument, after which counting the allowed sequences gives the total number of blockers.

Consider \(L(m,a)\), and suppose a nonnested perfect matching avoids its fan. Vertex \(1\) is paired with the first right endpoint, so this endpoint is greater than \(m+1\). Consequently the first \(m+1\) vertices are all left endpoints. If \(m=k\), this would require at least \(k+1\) left endpoints and is impossible.

Assume \(m<k\). For \(1\le i\le m+1\), we now have \(x_i=i\), so
\(h_i=y_i-i\) is the length of the rank-\(i\) edge. The right endpoints increase by at least one at each step, and every edge has length at most \(k\). Therefore
\[
 m+1\le h_1\le\cdots\le h_{m+1}\le k.
\]
The deleted edge of length \(j\) starts at \(a_j\). It lies in the matching exactly when \(h_{a_j}=j\), so our task is to find such a value of \(j\).

Define \(g(j)=h_{a_j}\) for \(m+1\le j\le k\). Both sequences \(a_j\) and \(h_i\) increase weakly, so \(g\) is weakly increasing and takes values in \(\intset{m+1}{k}\). Start at \(j_0=m+1\) and repeatedly set \(j_{s+1}=g(j_s)\). The first step satisfies \(j_1\ge j_0\), and monotonicity gives the same inequality at every subsequent step. A nondecreasing sequence in this finite interval eventually stops increasing, at a value \(j\) with \(g(j)=j\). The matching then contains the deleted edge \(\edge{a_j}{a_j+j}\).

For the running example \(L(2,a)\) on ten vertices, the matching
\(M=\{14,26,37,59,\edge{8}{10}\}\) avoids the fan \(\{12,13\}\).
Its first three lengths are \(h_1,h_2,h_3=3,4,4\), while
\(a_3,a_4,a_5=2,2,3\). The iteration is \(3\mapsto4\mapsto4\), and the fixed value \(4\) identifies the common edge \(26\). Fig.~\ref{blk:fig:fixed-point} displays the same calculation. Reflection proves sufficiency for every right blocker.

\begin{figure}[htbp]
\centering
\begin{adjustbox}{max width=\linewidth}
\begin{tikzpicture}[font=\small,
  lab/.style={font=\scriptsize,text height=1.6ex,text depth=.4ex,inner sep=1pt},
  mlab/.style={font=\scriptsize,inner sep=1pt},
  cell/.style={draw=gray!60,minimum width=8.6mm,minimum height=6.2mm,inner sep=0pt},
  del/.style={draw=oiverm,very thick,fill=oiverm!12,minimum width=8.6mm,minimum height=6.2mm,inner sep=0pt},
  mat/.style={draw=oiblue,thick,circle,minimum size=5.9mm,inner sep=0pt},
  vtx/.style={circle,fill=black,inner sep=1.15pt}]

\begin{scope}[x=0.95cm,y=0.74cm]
\node[anchor=west] at (-1.6,0.95) {(a)};
\node[lab] at (3.0,0.95) {left endpoint $a$};
\foreach \a in {1,...,5}{\node[lab] at (\a,0.35) {$\a$};}
\foreach \d in {1,...,5}{
  \node[lab,anchor=east] at (0.42,-\d) {$B_{0\d}$};
  \foreach \a in {1,...,\d}{\node[cell] at (\a,-\d) {};}}

\foreach \d/\a in {1/1,2/1,3/2,4/2,5/3}{\node[del] at (\a,-\d) {};}

\foreach \d in {1,...,5}{\foreach \a in {1,...,\d}{
  \pgfmathtruncatemacro{\b}{\a+\d}
  \ifnum\b<10 \node[font=\scriptsize] at (\a,-\d) {$\a\b$};
  \else \node[font=\scriptsize] at (\a,-\d) {$\{\a,\b\}$};\fi}}

\foreach \d/\a in {3/1,4/2,4/3}{\node[mat] at (\a,-\d) {};}
\draw[decorate,decoration={brace,amplitude=4pt,mirror}] (-0.45,-0.68)--(-0.45,-2.32)
  node[midway,left=4pt,mlab,align=right] {rows\\$d\le m$};
\draw[decorate,decoration={brace,amplitude=4pt,mirror}] (-0.45,-2.68)--(-0.45,-5.32)
  node[midway,left=4pt,mlab,align=right] {rows\\$d>m$};
\node[mlab,anchor=north west,align=left] at (0.45,-6.25) {
  \tikz\node[del,minimum width=3mm,minimum height=3mm]{};\ deleted edge $\{a_d,a_d{+}d\}$, one per row\\[1pt]
  \tikz\node[mat,minimum size=3mm]{};\ edge of $M$ of rank $i\le m{+}1$, in row $h_i$};
\end{scope}

\begin{scope}[xshift=7.05cm,yshift=-1.6cm,x=0.6cm]
\node[anchor=west] at (0.2,2.35) {(b)};
\draw[gray!45] (1,0)--(10,0);
\foreach \u/\w in {1/4,3/7,5/9,8/10}{\draw[thick,oiblue] (\u,0) to[bend left=55] (\w,0);}
\draw[very thick,oiblue] (2,0) to[bend left=55] (6,0);
\foreach \u/\w in {1/2,1/3,2/5,3/8}{\draw[thick,oiverm,dashed] (\u,0) to[bend right=55] (\w,0);}
\draw[very thick,oiverm,dashed] (2,0) to[bend right=55] (6,0);
\foreach \v in {1,...,10}{\node[vtx] at (\v,0) {};}
\foreach \v in {1,...,10}{\node[lab,fill=white,inner sep=0.5pt] at (\v,-0.3) {\v};}
\node[lab,anchor=west,oiblue] at (10.25,0.6) {$M$};
\node[lab,anchor=west,oiverm] at (10.25,-0.6) {$D$};
\node[mlab,anchor=north west,align=left] at (0.6,-1.75) {
  $h_1,h_2,h_3=3,4,4$\qquad $a_3,a_4,a_5=2,2,3$\\[2pt]
  $g(j)=h_{a_j}$:\ \ $3\mapsto h_2=4\mapsto h_2=4$\\[2pt]
  fixed point $j=4$: common edge $\{a_4,a_4{+}4\}=26$};
\end{scope}
\end{tikzpicture}
\end{adjustbox}
\caption{Why do increasing deletion positions force a common edge? In (a), a box marks the deleted edge of each length, and a circle marks one of the first three edges of the matching in (b). For $L(2,a)$ with $(a_3,a_4,a_5)=(2,2,3)$, the map $g$ sends $3$ to $4$ and fixes $4$. The common edge is therefore $\edge{a_4}{a_4+4}=26$. In (b), the matching edges are solid and the deletions are dashed.}\label{blk:fig:fixed-point}
\end{figure}

For \(C_m\), suppose a matching avoids both fans. Avoiding the left fan makes its first \(m+1\) left endpoints \(1,\ldots,m+1\). Avoiding the right fan makes its last \(k-m\) right endpoints \(k+m+1,\ldots,2k\). The first of these last right endpoints has rank \(m+1\), so the matching contains the central deleted edge \(\edge{m+1}{k+m+1}\). Every displayed form is therefore a blocker of size \(k\).

For a fixed \(m\), the sequence in Eq.~\eqref{blk:eq:left} chooses \(k-m\) entries with repetition from the \(m\) values \(2,\ldots,m+1\), and then places them in weakly increasing order. The number of such sequences is
\(\binom{(k-m)+m-1}{k-m}=\binom{k-1}{k-m}\).
Different choices give different blockers: the degree of vertex \(1\) recovers \(m\), and the unique deletion of each remaining length recovers its sequence entry. Summing for \(m=1,\ldots,k\) gives \(2^{k-1}\) left blockers, and reflection gives the same number of right blockers.

A left blocker contains \(12\) and omits \(\edge{2k-1}{2k}\); a right blocker has the reflected property; a mixed blocker contains both end edges. Thus the three classes are disjoint, and their sizes give Eq.~\eqref{blk:eq:count}. Taking complements retains \(\binom{2k}{2}-k=2k(k-1)\) edges and proves the equivalent extremal statement in Theorem~\ref{blk:thm:main}.

\section{Certificates and matrix connections}\label{blk:sec:consequences}
The classification can be used to answer concrete questions about a given deletion set. We first ask how few edges meet every minimum blocker, then turn the proof into an algorithm that returns either a blocker description or a matching surviving the deletions. Finally, we represent each left blocker by forbidden positions of a permutation matrix. These consequences use the classification to reverse the blocking question, construct explicit certificates, and transfer obstructions to a related family.

The displayed formulas already reveal the shape of each minimum blocker as a graph. After isolated vertices are removed, it is a tree on \(k+1\) vertices. Here the distance between two vertices is the smallest number of edges in a path joining them. In a left blocker, every nonfan edge joins a fan vertex to a distinct new vertex, so every vertex is at distance at most two from \(1\). For example, the edges \(13,38\) of our ten-vertex blocker give the two-edge path \(1,3,8\). A mixed blocker consists of the path \(1,m+1,k+m+1,2k\), with all additional vertices attached to its two outer vertices. Thus the diameter, the greatest distance between two vertices, is at most five in every case.

We can also reverse the blocking question. Let \(\mathcal B_k\) be the family of minimum blockers of \(\mathcal N_k\), and call an edge set \(T\) a \emph{transversal} of that family when \(T\cap D\ne\varnothing\) for every \(D\in\mathcal B_k\). Here the intersection must contain an edge: sharing only a vertex does not suffice. We now seek the smallest set that intersects every minimum deletion set. The classification shows that the two consecutive edges at the ends of the order are uniquely optimal.

\begin{proposition}\label{blk:prop:coblocker}
For \(k\ge3\), the family \(\BB\) of minimum blockers of \(\NN\)
has a unique transversal of minimum cardinality, namely
\(\{12,\edge{2k-1}{2k}\}\).
\end{proposition}
\begin{proof}
Every minimum blocker contains \(12\) or \(\edge{2k-1}{2k}\), so this pair meets all of \(\mathcal B_k\). The full left and right fans are disjoint minimum blockers. Consequently any transversal has at least two edges, and a two-edge transversal must contain one edge from each full fan.

To identify those two edges, consider
\[
 L(1,a)=\{12\}\cup\{\edge{2}{2+j}:2\le j\le k\},
\]
where \(a_j=2\) for every \(2\le j\le k\).
For \(k\ge3\), this blocker has no edge incident with \(2k\), so the edge chosen from the right fan cannot meet it. Its only edge incident with \(1\) is \(12\), forcing that choice from the left fan. Reflecting the argument forces \(\edge{2k-1}{2k}\) from the right fan. This proves both optimality and uniqueness.
\end{proof}

This transversal number also distinguishes the nonnested problem from the noncrossing problem. Place the \(2k\) vertices on a convex polygon in cyclic order, where a matching is noncrossing when the interiors of its chords are disjoint. The minimum blockers of noncrossing perfect matchings have transversal number \(k\), as follows from Keller and Perles~\cite{KP2013}. We include the short argument to make the comparison explicit.

Temporarily label the vertices \(0,\ldots,2k-1\). For each odd residue \(s\) modulo \(2k\), pair \(a\) with \(s-a\pmod{2k}\). No vertex is paired with itself, because \(2a\) cannot be odd modulo \(2k\). Each resulting matching is noncrossing: edges of the same ordinary endpoint sum are nested in the linear order, and the two possible sums \(s\) and \(s+2k\) occupy separated intervals. Distinct residues give edge-disjoint matchings. These \(k\) matchings force every blocker to contain at least \(k\) edges.

For every vertex \(v\), its \(k\) edges to vertices of opposite parity form a blocker. Indeed, a chord in a noncrossing perfect matching must have an even number of vertices on each side of it, so its endpoints have opposite parities. The matching's edge at \(v\) therefore meets this fan, which has minimum size \(k\). A transversal of all minimum blockers must meet the fan at every vertex and hence have an edge incident with every vertex; this requires at least \(k\) edges. Any fixed noncrossing perfect matching meets every blocker and attains the bound. This is the parallel-matching and star-blocker argument of~\cite{KP2012,KP2013}.

For \(k\ge3\), the two families of minimum blockers thus have different transversal numbers, \(2\) and \(k\). Any relabeling of the ground edges preserves intersections, blockers, and these transversal numbers. Hence no such relabeling identifies the noncrossing and nonnested matching families, even though both have the same Catalan cardinality.

\subsection{An algorithm with explicit certificates}\label{blk:sec:algorithm}
Given at most \(k\) deleted edges, we want an answer that can be checked directly. A positive answer identifies the parameters of one of the minimum blockers; a negative answer supplies all \(k\) edges of a surviving nonnested perfect matching. The classification proof already supplies these two kinds of certificate. We organize its constructions so that all tests take quadratic time and only linear additional storage.

For example, on ten vertices, \(D=\{12,13,25,26,38\}\) is certified by the description \(L(2,a)\), with \(a=(2,2,3)\). The block
\(B_{05}=\{16,27,38,49,\edge{5}{10}\}\) meets \(D\) only in the edge \(38\). If that edge is restored, all five edges of the block are available. They pair vertices \(1,\ldots,5\) with \(6,\ldots,10\) in increasing order, giving a surviving nonnested perfect matching. The algorithm first searches for a matching assembled from blocks \(B_{ij}\); only when this search fails does it inspect the equality forms.

\begin{theorem}\label{blk:thm:algorithm}
Given a set \(D\) of at most \(k\) edges on \(\intset{1}{2k}\), one can
deterministically return either a description of \(D\) in
Theorem~\ref{blk:thm:main} or a nonnested perfect matching disjoint from \(D\).
The algorithm uses \(O(k^2)\) time and \(O(k)\) auxiliary words
on a word-RAM with \(O(\log k)\)-bit integers.
\end{theorem}
In the computational model of the theorem, one word stores an endpoint, an index, or a counter, each using a number of bits proportional to \(\log k\). Algorithm~\ref{blk:alg:certificate} follows the three possibilities in the proof: a surviving auxiliary path, an internal split, or an end split. Its counters record how many deleted edges destroy a block of each length at the current starting index. The proof of the theorem defines the counter updates and shows why each returned object is a valid certificate.

\begin{algorithm}[htbp]
\caption{Construct a certificate for \(D\)}\label{blk:alg:certificate}
\begin{algorithmic}[1]
\Require \(k\ge1\); an edge set \(D\) on \(\intset{1}{2k}\) with \(|D|\le k\)
\Ensure A minimum-blocker description or an avoiding nonnested perfect matching
\If{\(k=1\)}
  \State \Return blocker \(\{12\}\) if \(12\in D\), otherwise matching \(\{12\}\)
\EndIf
\State Bucket the events from Eq.~\eqref{blk:eq:events}; set all length counters to zero\label{alg:events}
\State Mark only \(0\) reachable; initialize the predecessor array
\For{\(i=0,\ldots,k-1\)}
  \State Apply every event at \(i\) to its length counter
  \If{\(i\) is reachable}
    \For{\(d=1,\ldots,k-i\) with length counter zero}
      \State Mark \(i+d\) reachable and store its first predecessor
    \EndFor
  \EndIf
\EndFor
\If{\(k\) is reachable}
  \State \Return the matching formed by the predecessor-path blocks \(B_{ij}\)
\EndIf
\State Recover \(t\) and bucket deletions by length as in Eq.~\eqref{blk:eq:cutform}\label{alg:cut}
\If{\(1<t<k\)}
  \For{\(M=P_1,\rho(P_1),\ldots,P_{k-1},\rho(P_{k-1})\)}
    \If{\(M\cap D=\varnothing\)} \State \Return \(M\) \EndIf
  \EndFor
  \State \Return the mixed-blocker description \(C_{t-1}\)
\Else
  \State Reflect the order if \(t=1\)\label{alg:reflect}
  \State Scan lengths \(d=1,\ldots,k\), retaining the fan size and last sequence entry\label{alg:scan}
  \State An extension of a full fan is valid; update its fan or sequence parameters
  \State Otherwise test the monotone-extension condition of Lemma~\ref{blk:lem:extension}
  \If{the first invalid extension has length \(q\)}
    \State Construct its witness on \(\intset{1}{2q}\) from Lemma~\ref{blk:lem:extension}
    \State \Return the witness completed by consecutive pairs, reflecting back if needed
  \EndIf
  \State \Return the recovered left-blocker parameters, reflecting back if needed
\EndIf
\end{algorithmic}
\end{algorithm}

\begin{proof}
We first implement the auxiliary path search without storing all its arcs. A deletion \(\edge{a}{a+d}\) lies in \(B_{i,i+d}\) exactly when
\(2i+1\le a\le2i+d\). Combining these inequalities with \(0\le i\le k-d\) gives the interval of integer starts
\begin{equation}\label{blk:eq:events}
 \max\left(0,\left\lceil\frac{a-d}{2}\right\rceil\right)
 \le i\le
 \min\left(k-d,\left\lfloor\frac{a-1}{2}\right\rfloor\right).
\end{equation}
An edge longer than \(k\), or an empty interval in this formula, affects no auxiliary arc. For each nonempty interval \(\intset{l}{u}\), add an event \((d,+1)\) at \(l\) and an event \((d,-1)\) at \(u+1\). Line~\ref{alg:events} stores the at most \(2k\) events in lists indexed by \(0,\ldots,k\).

For illustration, when \(k=5\), deletion \(58\), of length \(3\), destroys the arcs \(1\longrightarrow4\) and \(2\longrightarrow5\). Eq.~\eqref{blk:eq:events} gives \(1\le i\le2\), so the length-\(3\) counter increases at start \(1\) and decreases at start \(3\). The same two events represent the entire interval of affected arcs, without listing its members separately.

Sweep through the starts \(i=0,\ldots,k-1\), applying every event at \(i\), including when \(i\) is unreachable. The counter for \(d\) then equals the number of deletions in \(B_{i,i+d}\), so that block survives exactly when the counter is zero. At every reachable \(i\), inspect \(1\le d\le k-i\); for each surviving block, mark \(i+d\) reachable and remember its first predecessor. There are at most \(k(k+1)/2\) inspections. If \(k\) is reached, following its predecessors gives a path whose blocks form the required avoiding matching.

Suppose \(k\) remains unreachable. Lemma~\ref{blk:lem:equality} now gives \(|D|=k\) and the form in Eq.~\eqref{blk:eq:cutform}. The split \(t\) is the least reachable positive interior index, or \(k\) when none exists. Store the smaller endpoints of the deletions in lists indexed by length. The equality form allows at most two entries for any length, so membership of a candidate edge in \(D\) can now be checked in constant time.

For an internal split, test \(P_p\) and \(\rho(P_p)\) for \(1\le p\le k-1\). The necessity proof in Subsection~\ref{blk:sec:classification} shows that either one of these matchings avoids \(D\) or every deletion is forced to give \(D=C_{t-1}\). There are fewer than \(2k\) candidate matchings, each with \(k\) edges, so these tests take \(O(k^2)\) time in total.

For an end split, line~\ref{alg:reflect} reverses the order if needed, and line~\ref{alg:scan} reads the unique deletion of each length in increasing order. At each stage retain the size of the fan and the last entry of the increasing sequence in Eq.~\eqref{blk:eq:left}, with a separate state for a full fan. If every extension satisfies Lemma~\ref{blk:lem:extension}, these parameters certify a left blocker.

If the first invalid extension has length \(q\), the same lemma constructs an avoiding matching on \(\intset{1}{2q}\). Append
\(\edge{2q+1}{2q+2},\ldots,\edge{2k-1}{2k}\)
to complete it on \(\intset{1}{2k}\). Every deletion not yet considered has length greater than \(q\), so the length bound excludes it from the prefix matching. The appended edges all have length one and differ from the sole length-one deletion \(12\). They are separated from the prefix, giving a nonnested perfect matching avoiding all of \(D\). Reflect the answer back when the initial end split required it.

The event lists, counters, predecessors, membership lists, and one candidate matching each occupy \(O(k)\) words. The end-split scan takes \(O(k)\) time, and the two quadratic stages establish the stated time bound. All positions, counters, and indices use \(O(\log k)\) bits. For \(k=1\), inspecting \(12\) directly returns either the blocker or the surviving one-edge matching.
\end{proof}

The assumption \(|D|\le k\) is used precisely after the path search fails. It turns the lower bound into equality and gives the split form, including the at-most-two entries per length. Larger deletion sets need not have that form, so the remaining classification tests have been justified only for the input range in Theorem~\ref{blk:thm:algorithm}.

\subsection{An embedding into matrix blockers}\label{blk:sec:matrix}
A left blocker chooses one deleted edge at each length, with a weakly increasing sequence of smaller endpoints. We now use this sequence to construct forbidden positions in a permutation matrix. The connection reuses the same monotone iteration as the blocking proof, followed by a count of available positions. A matrix has separate row and column copies of each index, which changes the disjointness condition. For example, the graph edges \(12,23\) share vertex \(2\), whereas the matrix positions \((1,2),(2,3)\) occupy distinct rows and columns. The construction below therefore needs an explicit map between the two settings.

A permutation \(\pi\) of \(\intset{1}{N}\) is represented by the \(N\) matrix positions \((i,\pi(i))\), one in every row and column. It \emph{avoids \(123\)} if there are no indices \(i<j<\ell\) with \(\pi(i)<\pi(j)<\pi(\ell)\). A \emph{\(123\)-blocker of order \(N\)} is a set of positions meeting every such avoiding permutation matrix. This is the minimum-blocker problem classified in~\cite{BC2023}. Its ordinary formulation allows a full row or full column as a blocker. Related problems require that some permutation matrix still fit in the allowed positions, a condition expressed as a nonzero permanent in~\cite{BCG2026}.

For completeness, the cyclic-Hankel construction gives the minimum size immediately~\cite[Lemma~2.2]{BC2023}. With indices \(0,\ldots,N-1\), the \(N\) sets
\(\{(r,s-r\bmod N):0\le r<N\}\), one for each residue \(s\), partition all matrix positions. Each set represents a permutation with at most two decreasing runs, and any three selected values have two in the same run, so it avoids \(123\). A blocker must meet all \(N\) disjoint sets and therefore contains at least \(N\) positions. A full row attains that size.

The same bound holds for transversals of the minimum matrix blockers. The \(N\) full rows are disjoint minimum blockers, so a transversal requires at least one position in every row. Conversely, any fixed \(123\)-avoiding permutation meets every blocker and uses exactly \(N\) positions. This transversal number, together with Proposition~\ref{blk:prop:coblocker}, distinguishes the full matrix problem from the nonnested problem when the parameters are at least three.

The next construction applies specifically to the left family. An edge \(\edge{a}{b}\) gives the two positions \((a,b)\) and \((b,a)\) before the columns are reversed; at length \(k\), only the first orientation is retained. One additional position brings the total to \(N=2k\), the minimum size just established. Fig.~\ref{blk:fig:matrix-embedding} illustrates this construction for the running example on ten vertices.

\begin{proposition}\label{blk:prop:matrix}
Put \(N=2k\), \(r(i)=N+1-i\), and write a left blocker as
\[
 D=\{\edge{a_d}{a_d+d}:1\le d\le k\},\qquad a_d=1\quad(d\le m).
\]
Then
\begin{align}
Z(D)={}&\{(1,N)\}
 \cup\{(a_d,r(a_d+d)),(a_d+d,r(a_d)):1\le d<k\}\notag\\
 &\cup\{(a_k,r(a_k+k))\}\label{blk:eq:embedding}
\end{align}
is a minimum \(123\)-blocker of order \(N\).
The map \(D\mapsto Z(D)\) is injective.
\end{proposition}

\begin{figure}[htbp]
\centering
\begin{adjustbox}{max width=\linewidth}
\begin{tikzpicture}[x=0.5cm,y=0.5cm,font=\scriptsize,
  lab/.style={font=\scriptsize,text height=1.6ex,text depth=.4ex,inner sep=1pt},
  mlab/.style={font=\scriptsize,inner sep=1pt},
  zc/.style={draw=black,minimum size=5mm,inner sep=0pt,font=\scriptsize}]

\foreach \i in {1,...,10}{\node[lab] at (-0.1,-\i) {\i};}
\foreach \c in {1,...,10}{\node[lab] at (\c,0.05) {\c};}
\node[lab] at (5.5,0.85) {column};
\node[lab,rotate=90] at (-1.0,-5.5) {row};
\foreach \x in {0,...,10}{\draw[gray!40] (\x+0.5,-0.5)--(\x+0.5,-10.5);}
\foreach \y in {0,...,10}{\draw[gray!40] (0.5,-\y-0.5)--(10.5,-\y-0.5);}

\foreach \r/\c/\col/\t in {1/10/black!25/{\ast},
   1/9/oiblue!45/{12},2/10/oiblue!45/{12},
   1/8/oiorange!55/{13},3/10/oiorange!55/{13},
   2/6/oigreen!45/{25},5/9/oigreen!45/{25},
   2/5/oipurple!50/{26},6/9/oipurple!50/{26},
   3/3/oiverm!45/{38}}{
  \node[zc,fill=\col] at (\c,-\r) {$\t$};}

\node[mlab,anchor=north west,align=left] at (11.3,-0.35) {
  $N=2k=10$,\ $r(i)=11-i$\\[3pt]
  edge of $D$ $\mapsto$ positions (row, column):\\[2pt]
  \tikz\node[zc,fill=black!25,minimum size=2.6mm]{};\ $\ast$ fixed position $(1,N)=(1,10)$, class $10$\\[1pt]
  \tikz\node[zc,fill=oiblue!45,minimum size=2.6mm]{};\ $12$ ($d=1$) $\mapsto(1,9),(2,10)$, classes $9,1$\\[1pt]
  \tikz\node[zc,fill=oiorange!55,minimum size=2.6mm]{};\ $13$ ($d=2$) $\mapsto(1,8),(3,10)$, classes $8,2$\\[1pt]
  \tikz\node[zc,fill=oigreen!45,minimum size=2.6mm]{};\ $25$ ($d=3$) $\mapsto(2,6),(5,9)$, classes $7,3$\\[1pt]
  \tikz\node[zc,fill=oipurple!50,minimum size=2.6mm]{};\ $26$ ($d=4$) $\mapsto(2,5),(6,9)$, classes $6,4$\\[1pt]
  \tikz\node[zc,fill=oiverm!45,minimum size=2.6mm]{};\ $38$ ($d=k$) $\mapsto(3,3)$, class $5$\\[4pt]
  class of $(i,j)$: $i+j-1 \bmod N$, residue $0$ read as $N$;\\
  one marked cell per class: $|Z(D)|=N$};
\end{tikzpicture}
\end{adjustbox}
\caption{Embedding the left blocker $D=\{12,13,25,26,38\}$ in a $10\times10$ matrix. Cell labels identify the original deleted edges. The edge $26$, for example, has orientations $(2,6)$ and $(6,2)$, which become positions $(2,5)$ and $(6,9)$ after reversing the columns. Every edge of length below five contributes two positions; the length-five edge $38$ contributes only $(3,3)$. The asterisk adds the fixed position $(1,10)$. These ten marked cells form $Z(D)$ and occupy distinct cyclic classes, giving $|Z(D)|=10$; Proposition~\ref{blk:prop:matrix} establishes that every $123$-avoiding permutation meets them.}\label{blk:fig:matrix-embedding}
\end{figure}
\begin{proof}
We first count the selected positions, then prove that they meet every \(123\)-avoiding permutation. Classify a position \((i,j)\) by the residue of \(i+j-1\) modulo \(N\), with residue zero written as \(N\). For \(d<k\), the two positions from the deletion of length \(d\) have classes \(N-d\) and \(d\). The extra position \((1,N)\) has class \(N\), and the remaining position from length \(k\) has class \(k\). Thus the construction has exactly one position in each of the \(N\) classes, so all selected positions are distinct.

Suppose a \(123\)-avoiding permutation avoids every selected position. Reverse its column values to obtain a permutation \(\pi\) avoiding \(321\), that is, having no decreasing subsequence of length three. This permutation avoids \((1,1)\), both orientations of every deleted edge of length below \(k\), and the forward orientation of the length-\(k\) deletion. We show that these exclusions force too many values into too few positions.

Assume first that \(m<k\). The exclusions from the fan give
\(\pi(1)>m+1\) and \(\pi^{-1}(1)>m+1\).
The first \(m+1\) values of \(\pi\) must increase: a descent among them, followed by the later value \(1\), would form \(321\). Because their first value exceeds \(m+1\), every one of these values does as well. The inverse permutation also avoids \(321\): decreasing inverse values at inputs \(u<v<w\) would place the decreasing values \(w,v,u\) at increasing positions in \(\pi\). Thus the same argument applies to the first \(m+1\) values of the inverse. Consequently
\[
 h_i=\pi(i)-i,\qquad v_i=\pi^{-1}(i)-i\qquad(1\le i\le m+1)
\]
are weakly increasing and are at least \(m+1\).

Let \(c=a_k\) and \(b=a_{k-1}\), so \(b\le c\). Starting from \(m+1\), iterate \(j\mapsto h_{a_j}\) as long as \(j\) lies in \(\intset{m+1}{k}\). The same monotonicity argument used in Subsection~\ref{blk:sec:classification} would give a fixed point if the iteration remained in that finite interval. Such a fixed point would be a forbidden forward deletion. The iteration must therefore leave through a value greater than \(k\). Every index \(a_j\) encountered before departure is at most \(c\), giving \(h_c\ge k+1\).

Apply the argument to the inverse values \(v_i\), this time using lengths \(m+1,\ldots,k-1\), whose reverse orientations were all deleted. It gives \(v_b\ge k\). When \(m=k-1\), this interval of lengths is empty, but \(b=1\) and the already established inequality \(v_1\ge m+1=k\) gives the same conclusion. Set \(x=\pi^{-1}(b)\) and \(y=\pi(c)\). The two displacement bounds imply \(x>c\) and
\[
 x+y\ge N+b+c+1.
\]

Now consider the values strictly between \(b\) and \(y\). None can occur in a position strictly between \(c\) and \(x\), because the values at positions \(c\), that intermediate position, and \(x\) would form the decreasing triple \(y,\pi(i),b\). The \(y-b-1\) intermediate values must therefore fit outside the entire interval \(\intset{c}{x}\), where only \(N-x+c-1\) positions are available. This requires
\[
 y-b-1\le N-x+c-1,\qquad\text{hence}\qquad x+y\le N+b+c,
\]
contradicting the preceding lower bound.

For a full fan, \(m=k\), the forward exclusions give \(y=\pi(1)\ge k+2\), while the reverse exclusions give \(x=\pi^{-1}(1)\ge k+1\). The same count, now with \(b=c=1\), gives \(x+y\le N+2\). But the two lower bounds sum to at least \(2k+3=N+3\), again a contradiction.

Thus \(Z(D)\) blocks every \(123\)-avoiding permutation, and its \(N\) positions attain the cyclic-Hankel lower bound. Finally, for each \(d<k\), its unique position in class \(N-d\) recovers the smaller endpoint \(a_d\); the unique position in class \(k\) recovers \(a_k\). Every original deletion can therefore be reconstructed from \(Z(D)\), proving injectivity.
\end{proof}

The map embeds the displayed left family into the matrix blockers and recovers each left blocker from its image. Its blocking proof uses the weak monotonicity of the sequence \(a_d\). The necessity argument in Subsection~\ref{blk:sec:classification} identifies every end-split nonnested blocker with this family after reflection when necessary; the full matrix problem also has other blockers. Thus the construction gives a precise correspondence for this family and leaves the broader incidence problems distinct.

\section{Extremal graphs when matchings can leave vertices unused}\label{blk:sec:beyond}
A \(k\)-edge matching on \(n>2k\) vertices can leave \(n-2k\) vertices unused. Those unused vertices allow it to avoid parts of a deletion set, so the perfect-matching classification no longer determines the minimum number of deletions. We first handle one unused vertex and obtain a graph with one edge more than \((k-1)n\). We then allow more unused vertices and coordinate the deletions at both ends to retain an additional number of edges proportional to \(k^2\), under explicit conditions on \(n\).

We use the \emph{excess} to mean \(e(G)-(k-1)n\) for a constructed graph \(G\), or \(f_k(n)-(k-1)n\) for the optimum. Thus an excess of one means exactly one retained edge above the proposed value \((k-1)n\). An edge is called \emph{short} if its endpoint difference is at most \(k-1\), and \emph{long} otherwise. The following operation adds \(k-1\) long edges with each new vertex, preserving both the absence of the target matching and the excess. It is also the extension described in~\cite[Section~2.2]{BFT2025}.

\begin{lemma}\label{ext:lem:ext-extension}
If \(G\) on \(\intset{1}{n}\), with \(n\ge2k-1\), has no nonnested
\(k\)-matching, adjoining a last vertex adjacent only to \(\intset{1}{k-1}\)
preserves this property and adds \(k-1\) long edges.
\end{lemma}
\begin{proof}
A new nonnested \(k\)-matching would have to use the added vertex, since the original graph contains none. The added vertex would be its largest right endpoint, paired with some \(i\le k-1\). Increasing endpoint order would then force all \(k\) left endpoints to belong to \(\intset{1}{i}\), which has at most \(k-1\) vertices. This is impossible, and every added edge is long because its length is at least \(n+1-(k-1)>k-1\).
\end{proof}

\subsection{The first positive excess}\label{blk:subsec:first-excess}
The first new case has \(n=2k+1\), so a \(k\)-matching leaves exactly one vertex unused. We delete fans at the two ends, forcing the used vertices near each end to have a particular endpoint role. Four further deletions then meet the matching at a rank determined by the position of its unused vertex. Counting the retained edges gives excess one for every \(k\ge5\), and the extension lemma carries this gain to all larger orders.

\begin{proposition}\label{blk:prop:construction}
For every \(k\ge5\) and \(n\ge2k+1\),
\[
 f_k(n)\ge(k-1)n+1.
\]
In particular, the equality proposed in
\cite[Conjecture~2.6]{BFT2025} fails throughout this range.
\end{proposition}
\begin{proof}
First take \(n=2k+1\) and delete
\begin{align*}
 D_k={}&\{1j:2\le j\le k-1\}
 \cup\{\edge{j}{2k+1}:k+3\le j\le2k\}\\
 &\cup\{ij:i\in\{3,4\},\ j\in\{k+3,k+4\}\}.
\end{align*}
The first two parts are fans, each with \(k-2\) edges. The last part joins the two vertices \(\{3,4\}\) to the two vertices \(\{k+3,k+4\}\), giving four edges. These parts are disjoint, so there are \(2k\) deletions and
\[
\binom{2k+1}{2}-2k=(k-1)(2k+1)+1
\]
retained edges. It remains to show that the retained graph has no nonnested \(k\)-matching.

For \(k=5\), Fig.~\ref{blk:fig:counterexample}(a) shows the ten deleted edges, leaving \(55-10=45\), one more than \(4\cdot11=44\). Fig.~\ref{blk:fig:counterexample}(b) lists the three possible locations of the unused vertex and the deleted edge forced in each case. We now establish those endpoint restrictions for every \(k\ge5\).

\begin{figure}[H]
\centering
\begin{minipage}{0.98\linewidth}
\centering
\begin{tikzpicture}[x=1.10cm,y=0.95cm,
  vertex/.style={circle,draw=black,fill=white,inner sep=0pt,
    minimum size=4.8mm,font=\small},
  fan/.style={draw=black,line width=0.8pt},
  rectangle edge/.style={draw=blue!55!black,line width=1pt,dashed}]
\node[font=\small] at (5,1.35) {(a) The ten deleted edges \(D_5\)};
\foreach \v in {1,...,11} {
  \node[vertex] (v\v) at ({\v-1},0) {\v};
}
\graph[use existing nodes,edges={fan,bend left=55}] {
  v1 -- {v2,v3,v4};
  {v8,v9,v10} -- v11;
};
\graph[use existing nodes,edges={rectangle edge}] {
  v3 -- [bend right=45] v8;
  v3 -- [bend right=60] v9;
  v4 -- [bend right=25] v8;
  v4 -- [bend right=45] v9;
};
\node[font=\small] at (5,-2.25)
  {\(6\) fan edges \(\,+\,4\) rectangle edges \(\,=\,10\) deletions};
\end{tikzpicture}

\smallskip
{\small (b) A matching avoiding both fans must use a rectangle edge\par}
\smallskip
{\small\renewcommand{\arraystretch}{1.18}
\begin{tabular}{c@{\qquad}c@{\qquad}c}
Unused vertex & Forced left endpoint & Forced right endpoint\\
\hline
\(2\le u\le10\) & \(x_3\in\{3,4\}\) & \(y_3\in\{8,9\}\)\\
\(u=1\) & \(x_2\in\{3,4\}\) & \(y_2=8\)\\
\(u=11\) & \(x_4=4\) & \(y_4\in\{8,9\}\)
\end{tabular}\par}
\medskip
\(\displaystyle |E(K_{11}-D_5)|=\binom{11}{2}-10=45>44=4\cdot11.\)
\end{minipage}
\caption{A counterexample on eleven ordered vertices.
In (a), we show all ten deleted edges of \(D_5\), with vertices ordered
from left to right. The six solid edges
form the endpoint fans, and the four dashed edges join
\(\{3,4\}\) to \(\{8,9\}\).
In (b), \(u\) is the unused vertex of a hypothetical nonnested
five-edge matching that avoids both fans; its endpoints are indexed in
increasing order. Each case forces a dashed edge, so the complement
\(K_{11}-D_5\) has 45 edges and contains no such matching.}
\label{blk:fig:counterexample}
\end{figure}

Suppose such a matching exists, written as \(x_i y_i\), \(1\le i\le k\), with both endpoint lists increasing. Let \(u\) be its unique unused vertex. If \(1\) is used, then \(x_1=1\), and avoiding the left fan gives \(y_1\ge k\). Therefore every used vertex in \(\intset{1}{k-1}\) is a left endpoint. If \(2k+1\) is used, then \(y_k=2k+1\), and avoiding the right fan gives \(x_k\le k+2\). Therefore every used vertex in \(\intset{k+3}{2k+1}\) is a right endpoint.

If both end vertices are used, at least three vertices among \(1,2,3,4\) are used left endpoints, so \(x_3\in\{3,4\}\). At least \(k-2\) used vertices lie in \(\intset{k+3}{2k+1}\), all as right endpoints. At most two right endpoints therefore precede \(k+3\), giving \(y_3\ge k+3\). There must also be \(k-3\) right endpoints after \(y_3\), which gives \(y_3\le k+4\). Thus the rank-three edge is one of the four additional deletions.

If \(u=1\), all \(k-1\) vertices of \(\intset{k+3}{2k+1}\) are right endpoints. There is only one other right endpoint, so \(y_2=k+3\), and at most one of the first three used vertices \(2,3,4\) is a right endpoint. Consequently \(x_2\in\{3,4\}\), and the rank-two edge is deleted.

Finally, if \(u=2k+1\), all \(k-1\) vertices of \(\intset{1}{k-1}\) are left endpoints. Since \(k\ge5\), this gives \(x_4=4\). Exactly one vertex in \(\intset{k}{2k}\) is a left endpoint, while its other \(k\) vertices are the right endpoints. Removing that one left endpoint from the increasing list \(k,k+1,\ldots,2k\) places the fourth right endpoint at \(k+3\) or \(k+4\). The rank-four edge is again deleted.

Every position of \(u\) has forced an edge of \(D_k\), proving the blocking property. Lemma~\ref{ext:lem:ext-extension} now adds \(k-1\) edges for each additional vertex and preserves excess one for every \(n\ge2k+1\). This construction will also appear as the case \(a=b=1\), \(c=0\) of Theorem~\ref{ext:thm:coordinated-construction}.
\end{proof}

The construction gives a lower bound on the number of edges that can be retained. To obtain an upper bound near \(n=2k+1\), we return to the perfect-blocker classification: removing one vertex leaves \(2k\) vertices, where at least \(k\) deletions are necessary. We combine that fact with a restriction on how evenly a blocker can distribute its deleted edges. Write \(\Delta(D)\) for its maximum degree, the greatest number of deleted edges incident with a vertex.

\begin{lemma}\label{blk:lem:degree-two}
If \(k\ge2\) and \(D\) is an edge set on \(\intset{1}{2k+1}\) with
\(\Delta(D)\le2\), then its complement contains a nonnested
matching of size \(k\).
\end{lemma}
\begin{proof}
Call an edge \emph{available} when it is outside \(D\). We build the matching by induction, beginning with five vertices. Suppose first that \(12\) is available. It can be combined with any available edge on \(\{3,4,5\}\); if that entire triangle is deleted, its vertices already have two deleted neighbors, so \(13\) and \(24\) are available and form a nonnested pair.

If \(12\in D\) and \(13\) is available, combine \(13\) with an available edge on \(\{2,4,5\}\). Such an edge must exist, because otherwise vertex \(2\) would have the three deleted neighbors \(1,4,5\). If both \(12\) and \(13\) are deleted, the degree bound makes \(14\) available. Combine it with \(25\) or \(35\) if possible. If both are deleted, the degree bound makes \(23\) and \(45\) available. Each of the indicated pairs is nonnested, completing the base case.

Now let \(2k+1\ge7\). If \(12\) or \(13\) is available, remove its endpoints and apply induction to the remaining \(2k-1\) vertices in their inherited order. The deleted-edge degree bound is preserved, and the selected edge cannot nest any disjoint edge. It also cannot be nested by an edge because its left endpoint is \(1\). Appending it therefore gives the required \(k\)-matching.

The remaining case has \(12,13\in D\) and \(14\) available. Remove vertices \(1,4\) and also forbid \(23\) in the smaller problem. Vertices \(2,3\) have each lost a deleted edge to \(1\), so adding this prohibition keeps the deleted-edge degree at most two. Induction gives a nonnested matching avoiding \(23\). Its only possible edge nested inside \(14\) would be \(23\), and no edge can surround \(14\) from the left. Appending \(14\) therefore completes a nonnested \(k\)-matching.
\end{proof}

\begin{samepage}
\begin{proposition}\label{blk:prop:extra-upper}
For \(k\ge2\),
\[
 f_k(2k+1)\le2k^2-3.
\]
For \(k\ge10\), this improves to
\[
 f_k(2k+1)\le2k^2-4.
\]
Consequently, for \(k\ge10\) and \(n\ge2k+1\),
\[
 f_k(n)\le
 \left\lfloor\frac{(2k^2-4)n(n-1)}{2k(2k+1)}\right\rfloor.
\]
\end{proposition}
\end{samepage}
\begin{proof}
Let \(D\) block every nonnested \(k\)-matching on \(\intset{1}{2k+1}\). Removing any vertex \(v\) leaves a blocker \(D-v\) of the perfect matchings in the inherited order. Theorem~\ref{blk:thm:main} therefore gives
\(|D|-\deg_D(v)\ge k\).
Lemma~\ref{blk:lem:degree-two} shows that some vertex has deleted-edge degree at least three. Hence \(|D|\ge k+3\), and taking complements gives
\(f_k(2k+1)\le\binom{2k+1}{2}-(k+3)=2k^2-3\).

For the improvement, suppose \(k\ge10\) and \(|D|=k+3\). The same inequality gives \(\Delta(D)\le3\). Delete a vertex of degree three; the remaining set is a minimum perfect blocker with \(k\) edges and maximum degree at most three. If it is a left blocker \(L(m,a)\), then its fan has \(m\le3\) edges, and each of its \(m\) fan neighbors supports at most two further edges. Thus \(k\le m+2m\le9\), a contradiction. Reflection handles the right family. For a mixed blocker, the two outer degrees \(m\) and \(k-m-1\) are at most three, giving \(k\le7\), again a contradiction. Therefore \(|D|\ge k+4\), which proves \(f_k(2k+1)\le2k^2-4\).

Finally, take any graph on \(\intset{1}{n}\) avoiding a nonnested \(k\)-matching and apply the last bound to each induced subgraph on \(2k+1\) vertices, obtained by keeping all original edges between the chosen vertices. There are \(\binom{n}{2k+1}\) such subgraphs, and each edge appears in \(\binom{n-2}{2k-1}\) of them. Counting pairs consisting of an edge and a containing subgraph gives
\[
 e(G)\binom{n-2}{2k-1}
 \le (2k^2-4)\binom{n}{2k+1}.
\]
Dividing by the left binomial coefficient and using that \(e(G)\) is an integer gives the stated bound for general \(n\).
\end{proof}

Propositions~\ref{blk:prop:construction} and~\ref{blk:prop:extra-upper} give
two-sided bounds on \(2k+1\) vertices. Determining the exact
value of \(f_k(2k+1)\) for general \(k\), and classifying its extremal
graphs, remain further problems.
The averaged bound grows quadratically in $n$ for fixed $k$, so it is useful only near $2k+1$. For example, at $k=10$ it improves the general linear bound for $n=21,22,23$, whereas from $n=24$ that linear bound is smaller. We now enlarge the excess by coordinating the same endpoint restrictions, and compare the construction with the general short-edge and long-edge upper bounds.

\subsection{Quadratic excess from coordinated endpoint fans}
\label{ext:sec:ext-quadratic}\label{ext:subsec:coordinated-extremal}
The previous construction retains one extra edge for every \(k\ge5\). We now let the sets near the two ends grow with \(k\), so that the number of extra edges can grow proportionally to \(k^2\). The known bounds
\((k-1)n\le f_k(n)\le(k-1)n+\binom{k-1}{2}\)
for \(n\ge2k\)~\cite[Theorem~2.5]{BFT2025} permit this scale of improvement. Our construction attains it by keeping as many long edges as any avoiding graph can have, while retaining additional short edges.

For the counts, write \(r=k-1\), the largest permitted size of a nonnested matching. Following the decomposition in~\cite[Section~2.2.1]{BFT2025}, let \(S(G)\) and \(L(G)\) be the numbers of short and long edges, respectively. It is useful to compare them with
\begin{equation}\label{ext:eq:ext-reference-values}
 S_0(r)=\frac{3r^2+r}{2},\qquad L_0(n,r)=rn-S_0(r).
\end{equation}
These are the counts in a familiar graph with exactly \(rn\) edges: start with the complete graph on \(2r+1=2k-1\) vertices and apply Lemma~\ref{ext:lem:ext-extension}. Its initial short-edge count is
\(\sum_{d=1}^{r}(2r+1-d)=(3r^2+r)/2\), and every added edge is long. Proposition~\ref{ext:prop:ext-upper-bounds} proves that \(L_0(n,r)\) is the largest possible long-edge count, using the published theorem of Capoyleas and Pach~\cite{CP1992}. Thus increasing \(S(G)\) while preserving \(L_0(n,r)\) gives exactly the same increase in the total edge count.

The construction has three main parameters. We select \(a\) vertices at the left end and \(b\) at the right, and start with \(N=2k+a+b-1\) vertices, leaving \(a+b-1\) unused vertices in any hypothetical \(k\)-matching. At each selected vertex, we delete a group of incident edges, giving a union of fans at each end. These deletions restrict the possible endpoint ranks. To describe the remaining deletions, represent an edge \(xy\), \(x<y\), by a cell in row \(x\) and column \(y\). Two further sets of cells form a rectangle and a triangle in this array. A parameter \(c\) extends the endpoint deletions and narrows the rectangle; the triangle contains \(\binom{c+1}{2}\) cells and handles the matchings that avoid every row of the rectangle. The next theorem gives the retained-edge counts, and the following subsection specifies and proves the construction. As a running example, take $k=8$, $a=b=2$ and $c=1$. The initial graph then has $N=19$ vertices, so a target matching leaves three vertices unused; the construction will retain 136 edges, three above the proposed value $7\cdot19=133$.

\begin{theorem}\label{ext:thm:coordinated-construction}
Let \(a,b\ge1\), \(0\le c\le\min(a,b)\), \(t=a+b\), and \(w=t-c\)
be integers, with \(k\ge3w-1\). For every \(n\ge2k+t-1\), there is
an ordered graph \(G\) on \(\intset{1}{n}\) with no nonnested \(k\)-matching and
\begin{equation}\label{ext:eq:coordinated-counts}
 e(G)=rn+\xi,\qquad S(G)=S_0(r)+\xi,\qquad
 L(G)=L_0(n,r),\qquad \xi=ab-\binom{c+1}{2}.
\end{equation}
\end{theorem}

The proof is developed in Subsection~\ref{ext:subsec:coordinated-blocker}: Lemma~\ref{ext:lem:coordinated-blocker} proves the blocking property, and the concluding count starts on page~\pageref{ext:proof:coordinated-construction}, followed by the extension to larger \(n\). The quantity \(\xi=ab-\binom{c+1}{2}\) is the number of extra retained edges above \(rn\), and the theorem places all of them in the short-edge count. The condition \(k\ge3w-1\) ensures room for the chosen intervals, while \(n\ge2k+t-1\) allows their \(t-1\) unused vertices.

\subsection{Endpoint ranks and the two further deletion regions}
\label{ext:subsec:coordinated-blocker}
The two unions of fans will force used vertices in an initial interval to be left endpoints and used vertices in a final interval to be right endpoints. We place a rectangle of deleted edges between those intervals and compare the ranks appearing on its two sides. The rectangle catches every remaining matching except a case in which all its left-side vertices are unused. A small triangle catches that case, so the proof must account for the unused vertices as well as the endpoint ranks.

Let \(a,b\ge1\), \(0\le c\le\min(a,b)\), \(t=a+b\), and \(w=t-c\), with \(k\ge3w-1\). We construct the graph first on \(N=2k+t-1\) vertices. Define four interval endpoints by
\begin{equation}\label{ext:eq:coordinated-positions}
 p=k-b+c,\qquad q=k+a+w,\qquad
 h=a+w,\qquad f=a+2w-1.
\end{equation}
The six intervals below organize the four deletion regions. The example uses \(k=8\), \(a=b=2\), \(c=1\), hence \(t=4\), \(w=3\), \(N=19\), and \(p=f=7\), \(q=13\), \(h=5\).
\begin{center}
\small
\begin{tabular}{@{}lp{6.2cm}l@{}}
\toprule
Set & Role & Example\\
\midrule
\(A=\intset{1}{a}\) & Leftmost vertices, at which the left fans start & \(\{1,2\}\)\\
\(B=\intset{N-b+1}{N}\) & Rightmost vertices, at which the right fans end & \(\{18,19\}\)\\
\(P=\intset{1}{p}\) & Interval containing the left fans & \(\intset{1}{7}\)\\
\(Q=\intset{q}{N}\) & Interval containing the right fans & \(\intset{13}{19}\)\\
\(I=\intset{h}{f}\) & Row vertices of the rectangle & \(\{5,6,7\}\)\\
\(J=\intset{q}{q+t-1}\) & Column vertices of the rectangle & \(\intset{13}{16}\)\\
\bottomrule
\end{tabular}
\end{center}
A cell of the rectangle represents an edge joining a row vertex in \(I\) to a column vertex in \(J\). These sets have \(|I|=w\) and \(|J|=t\), and the assumption \(c\le\min(a,b)\) gives \(w\ge\max(a,b)\).

We check that all these positions fit before defining the deletions. The parameter assumptions give
\begin{equation}\label{ext:eq:coordinated-geometry}
 p-f=k-3w+1\ge0,\qquad q-p=2w,\qquad
 N-\max J=k-a-w\ge0.
\end{equation}
In particular, \(I\subseteq P\setminus A\) and \(J\subseteq Q\). The intervals \(J\) and \(B\) can meet only at an endpoint, because
\[
 \min B-\max J=k-2w-c+1\ge w-c\ge0.
\]
Equality throughout occurs precisely when \(k=3w-1\) and \(a=b=c=w\); then their intersection is the first vertex of \(B\). The proof below includes this possible shared vertex.

Delete every edge inside \(P\) with smaller endpoint in \(A\), and every edge inside \(Q\) with larger endpoint in \(B\). Write these two unions of fans as
\begin{equation}\label{ext:eq:coordinated-fans}
 F_A=\{xy:1\le x\le a,\ x<y\le p\},\qquad
 F_B=\{xy:q\le x<y\le N,\ y\in B\}.
\end{equation}
Also delete the rectangle and triangle
\begin{equation}\label{ext:eq:coordinated-repair}
 R=I\times J,\qquad
 \Theta_c=\{\edge{a-c+i}{q+j}:1\le i\le c,\ 0\le j<i\}.
\end{equation}
Here \(I\times J\) denotes all edges joining the two intervals. The triangular array uses the last \(c\) vertices of \(A\); its \(i\)th row has \(i\) edges into the start of \(Q\). When \(c=0\), it is empty. Both regions join \(P\) to \(Q\), and their smaller endpoints lie in the disjoint sets \(I\) and \(A\). Each of the two unions lies entirely within \(P\) or entirely within \(Q\), so the four parts of
\begin{equation}\label{ext:eq:coordinated-D}
 D=F_A\cup F_B\cup R\cup\Theta_c
\end{equation}
are pairwise disjoint.

For the running example, the left union deletes the six edges from \(1\) to \(2,\ldots,7\) and the five from \(2\) to \(3,\ldots,7\), for eleven deletions. The right union has eleven by reflection of the interval sizes. The rectangle contributes \(3\cdot4=12\) edges, and the triangle is the single edge \(\edge{2}{13}\). Thus \(D\) has \(11+11+12+1=35\) edges, and its complement has
\[
 \binom{19}{2}-35=136=7\cdot19+3
\]
edges. Fig.~\ref{ext:fig:coordinated-fans} locates every deletion. There are 25 short deletions and ten long deletions, leaving 80 short and 56 long edges, in agreement with \(S_0(7)+3=80\) and \(L_0(19,7)=56\).

\begin{figure}[htbp]
\centering
\begin{tikzpicture}[x=0.36cm,y=0.33cm,font=\scriptsize]

\draw[decorate,decoration={brace,mirror,amplitude=3pt}] (-2,-.5)--(-2,-2.5) node[midway,left=3pt] {$A$};
\draw[decorate,decoration={brace,mirror,amplitude=3pt}] (-2,-4.5)--(-2,-7.5) node[midway,left=3pt] {$I$};
\draw[decorate,decoration={brace,mirror,amplitude=3pt}] (-4,-.5)--(-4,-7.5) node[midway,left=3pt] {$P$};
\draw[decorate,decoration={brace,amplitude=3pt}] (12.5,1.4)--(16.5,1.4) node[midway,above=3pt] {$J$};
\draw[decorate,decoration={brace,amplitude=3pt}] (17.5,1.4)--(19.5,1.4) node[midway,above=3pt] {$B$};
\draw[decorate,decoration={brace,amplitude=3pt}] (12.5,3.1)--(19.5,3.1) node[midway,above=3pt] {$Q$};
\foreach \i in {1,...,19}{
 \node[left] at (0.35,-\i) {\i};
 \node[above] at (\i,-0.35) {\i};
 \foreach \j in {1,...,19}{
  \ifnum\j>\i
   \pgfmathtruncatemacro{\ell}{\j-\i}
   \ifnum\ell>7
    \fill[gray!15] (\j-0.5,-\i-0.5) rectangle (\j+0.5,-\i+0.5);
   \fi
   \draw[gray!35,line width=0.15pt] (\j-0.5,-\i-0.5) rectangle (\j+0.5,-\i+0.5);
  \fi
 }
}
\foreach \i in {1,2}{\foreach \j in {2,...,7}{\ifnum\j>\i
 \fill (\j-0.33,-\i-0.33) rectangle (\j+0.33,-\i+0.33);\fi}}
\foreach \j in {18,19}{\foreach \i in {13,...,18}{\ifnum\j>\i
 \fill (\j-0.33,-\i-0.33) rectangle (\j+0.33,-\i+0.33);\fi}}
\foreach \i in {5,6,7}{\foreach \j in {13,14,15,16}{
 \draw[thick,blue!55!black,fill=white] (\j,-\i) circle[radius=0.31];}}
\node[regular polygon,regular polygon sides=3,draw=black,fill=gray!55,inner sep=1.35pt] at (13,-2) {};
\draw[densely dashed,thick] (8.5,-1)--(19,-11.5);
\draw[thick,blue!55!black] (12.5,-4.5) rectangle (16.5,-7.5);
\node[above] at (14.5,-4.4) {$R=I\times J$};
\node at (3.7,-4.0) {$F_A$};
\node[left] at (16.8,-15.8) {$F_B$};
\node[left] at (12.45,-2) {$\Theta_1$};
\begin{scope}[shift={(23,-3)}]
\fill (-0.32,-0.32) rectangle (0.32,0.32);
\node[right] at (0.9,0) {fan edges: $11+11$};
\draw[thick,blue!55!black,fill=white] (0,-2) circle[radius=0.31];
\node[right] at (0.9,-2) {rectangle: $12$};
\node[regular polygon,regular polygon sides=3,draw=black,fill=gray!55,inner sep=1.35pt] at (0,-4) {};
\node[right] at (0.9,-4) {triangle: $1$};
\fill[gray!15] (-0.5,-6.5) rectangle (0.5,-5.5);
\node[right] at (0.9,-6) {long edge, $j-i\ge8$};
\node[align=left,anchor=west] at (-0.5,-10) {$|D|=35$\\$e(K_{19}-D)=136$\\$(k-1)N=133$};
\end{scope}
\end{tikzpicture}
\caption{Deleted edges in the coordinated construction for \(k=8\), \(a=b=2\), and \(c=1\). The cell in row \(i\), column \(j\) represents edge \(ij\); marked cells form the blocker \(D\). The fans end at \(p=7\) and begin at \(q=13\), the rectangle joins \(I=\{5,6,7\}\) to \(J=\intset{13}{16}\), and the triangle consists of edge \(\edge{2}{13}\). The 25 short deletions leave 80 short edges, three above \(S_0(7)=77\); the 10 long deletions leave 56 long edges and attain the long-edge bound.}\label{ext:fig:coordinated-fans}
\end{figure}

Fig.~\ref{ext:fig:triangle-mechanism} displays the four cases in the blocking argument on the same nineteen vertices. A set is called used here when at least one of its vertices belongs to the matching; individual vertices in that set may still be unused. The first three cases force an edge of the rectangle. In the last, all its rows are unused, and the triangle meets the matching at a different rank. This explains why both additional deletion regions are needed.

\begin{figure}[htbp]
\centering
\begin{adjustbox}{max width=\linewidth}
\begin{tikzpicture}[x=0.6cm,y=1cm,font=\small,
  lab/.style={font=\scriptsize,text height=1.6ex,text depth=.4ex,inner sep=1pt},
  lft/.style={circle,fill=black,inner sep=1.5pt},
  rgt/.style={circle,draw=black,fill=white,inner sep=1.3pt},
  unu/.style={cross out,draw=oiverm,thick,inner sep=1.6pt},
  forced/.style={very thick,oiverm}]

\fill[oigreen!14] (4.6,-4.95) rectangle (7.4,0.95);
\fill[oiblue!12] (12.6,-4.95) rectangle (16.4,0.95);
\foreach \v in {1,...,19}{\node[lab] at (\v,0.75) {\v};}
\draw[decorate,decoration={brace,amplitude=3pt}] (0.8,1.15)--(2.2,1.15) node[midway,above=2pt,lab] {$A$};
\draw[decorate,decoration={brace,amplitude=3pt}] (0.8,1.75)--(7.2,1.75) node[midway,above=4pt,lab] {$P=\intset{1}{p}$, $p=7$};
\draw[decorate,decoration={brace,amplitude=3pt}] (4.8,1.15)--(7.2,1.15) node[midway,above=2pt,lab] {$I$};
\draw[decorate,decoration={brace,amplitude=3pt}] (12.8,1.15)--(16.2,1.15) node[midway,above=2pt,lab] {$J$};
\draw[decorate,decoration={brace,amplitude=3pt}] (12.8,1.75)--(19.2,1.75) node[midway,above=4pt,lab] {$Q=\intset{q}{N}$, $q=13$};
\draw[decorate,decoration={brace,amplitude=3pt}] (17.8,1.15)--(19.2,1.15) node[midway,above=2pt,lab] {$B$};

\foreach \y/\L/\R/\U/\fu/\fw/\cond/\res in {
  0/{3,4,5,8,9,10,11,12}/{6,7,13,14,15,16,17,18}/{1,2,19}/5/13/{(a) $A$ unused, $B$ used}/{rank $w=3$ lies in $R$},
  1/{1,2,3,4,5,6,13,14}/{8,9,10,11,12,15,16,17}/{7,18,19}/6/15/{(b) $B$ unused, $A$ used}/{rank $2w=6$ lies in $R$},
  2/{1,2,3,4,5,6,7,8}/{9,10,11,12,13,17,18,19}/{14,15,16}/5/13/{(c) $A,B,I$ used}/{common rank $5$ lies in $R$},
  3/{1,2,3,4,8,9,10,11}/{12,13,14,15,16,17,18,19}/{5,6,7}/2/13/{(d) $A,B$ used, $I$ unused}/{rank $s=2$ lies in $\Theta_1$}}{
  \begin{scope}[yshift=-\y*1.45cm]
  \draw[gray!50] (1,0)--(19,0);
  \draw[forced] (\fu,0) to[bend left=13] (\fw,0);
  \foreach \v in \L {\node[lft] at (\v,0) {};}
  \foreach \v in \R {\node[rgt] at (\v,0) {};}
  \foreach \v in \U {\node[unu] at (\v,0) {};}
  \node[lab,anchor=north west] at (0.6,-0.22) {\cond};
  \node[lab,anchor=west] at (19.5,0) {\res};
  \end{scope}}

\begin{scope}[yshift=-5.5cm]
\node[lft] at (1,0) {}; \node[lab,anchor=west] at (1.3,0) {left endpoint};
\node[rgt] at (5,0) {}; \node[lab,anchor=west] at (5.3,0) {right endpoint};
\node[unu] at (9,0) {}; \node[lab,anchor=west] at (9.3,0) {unused};
\draw[forced] (12.4,0)--(13.4,0); \node[lab,anchor=west] at (13.6,0) {the only edge of the row's matching in $D$};
\end{scope}
\end{tikzpicture}
\end{adjustbox}
\caption{The four cases of Lemma~\ref{ext:lem:coordinated-blocker}, for $k=8$, $a=b=2$, and $c=1$. A filled dot is a left endpoint, an open dot is a right endpoint, and a cross is unused. Pair endpoints in increasing rank in each row. In (a) and (b), ranks $w=3$ and $2w=6$ respectively force the emphasized rectangle edge. In (c), the rank intervals of $I$ and $J$ intersect at rank five. In (d), every vertex of $I$ is unused and rank two forces the triangular edge $\edge{2}{13}$. Each row meets $D$ in exactly its emphasized edge.}\label{ext:fig:triangle-mechanism}
\end{figure}

\begin{lemma}\label{ext:lem:coordinated-blocker}
The set \(D\) in Eq.~\eqref{ext:eq:coordinated-D} meets every
nonnested \(k\)-matching on \(\intset{1}{N}\).
\end{lemma}

\begin{proof}
Let \(M=\{x_i y_i:1\le i\le k\}\) avoid both \(F_A\) and \(F_B\), with both endpoint lists increasing. A set of vertices will be called \emph{used} if at least one of its vertices belongs to \(M\), and \emph{unused} if none does. Let \(U\) be the set of unused vertices and write \(u(W)=|U\cap W|\) for any vertex set \(W\). Because \(|U|=t-1=a+b-1\), the two extreme sets \(A,B\), whose total size is \(a+b\), cannot both be unused.

If \(A\) is used, its smallest used vertex is \(x_1\). Avoiding \(F_A\) forces \(y_1>p\), so every used vertex of \(P\) is a left endpoint. Similarly, if \(B\) is used, its largest used vertex is \(y_k\), and avoiding \(F_B\) forces \(x_k<q\). Every used vertex of \(Q\) is then a right endpoint. We use these restrictions in the four cases of Fig.~\ref{ext:fig:triangle-mechanism}, forcing an edge in \(R\) in the first three cases and an edge in \(\Theta_c\) in the last.

\medskip\noindent\emph{Case 1: \(A\) is unused.}
Then \(B\) is used, so all \(k\) left endpoints precede \(q\). Since the first \(a\) vertices are unused, the \(w\)th left endpoint is at least \(a+w\). There must also be \(k-w\) later left endpoints before \(q\). Hence
\[
 a+w\le x_w\le q-1-(k-w)=a+2w-1,
\]
which places \(x_w\) in \(I\).

To locate the vertex \(y_w\) matched to \(x_w\), note that \(|Q|=k-a+c\), and all used vertices of \(Q\) are right endpoints. The number of right endpoints before \(q\) is therefore \(a-c+u(Q)\). All \(a\) vertices of \(A\) are already unused, leaving at most \(b-1\) unused vertices elsewhere, so
\[
 a-c+u(Q)\le a-c+b-1=w-1.
\]
Thus \(y_w\ge q\). The \(k-w\) right endpoints after \(y_w\) give the other bound
\[
 y_w\le N-(k-w)=k+t+w-1\le q+t-1.
\]
We have \(x_w\in I\) and \(y_w\in J\), so their edge is deleted by \(R\).

\medskip\noindent\emph{Case 2: \(B\) is unused.}
Then \(A\) is used, and at most \(a-1\) unused vertices lie outside \(B\). Every used vertex through \(f\) is a left endpoint because \(\intset{1}{f}\subseteq P\). There are at least
\[
 f-(a-1)=2w
\]
such endpoints. The rank \(2w\) exists since \(k\ge3w-1\ge2w\), and
\(x_{2w}\ge2w\ge a+w=h\). Together with \(x_{2w}\le f\), this places it in \(I\).

Every right endpoint before \(q\) lies outside \(P\), where only \(q-1-p=2w-1\) positions are available. Hence \(y_{2w}\ge q\). Since \(B\) is unused, all right endpoints are at most \(N-b\), and the \(k-2w\) endpoints after this one give
\[
 q\le y_{2w}\le N-b-(k-2w)
              =k+a+2w-1\le q+t-1.
\]
Thus the edge of rank \(2w\) also belongs to \(R\).

\medskip\noindent\emph{Case 3: \(A\), \(B\), and \(I\) are used.}
Both endpoint restrictions apply: used vertices of \(I\) are left endpoints, and used vertices of \(J\) are right endpoints. We compare their ranks, because any common rank pairs a vertex of \(I\) with a vertex of \(J\). The left-endpoint ranks in \(I\) form the nonempty integer interval \(\intset{\ell_0+1}{\ell_1}\), where
\[
 \ell_0=h-1-u(\intset{1}{h-1}),\qquad
 \ell_1=f-u(\intset{1}{f}).
\]
These two counts give the number of left endpoints before \(I\) and through \(I\), respectively. For the right endpoints in \(J\), the corresponding interval is \(\intset{\eta_0+1}{\eta_1}\), where
\[
 \eta_0=a-c+u(Q),\qquad
 \eta_1=h+u(Q\setminus J).
\]
The first count is the number of right endpoints before \(q\); adding the \(t-u(J)\) used vertices in \(J\) gives the second. This interval is nonempty because \(|J|=t>|U|\).

Two nonempty integer intervals of these forms intersect when
\(\ell_1>\eta_0\) and \(\eta_1>\ell_0\). For the first inequality, the disjointness of \(\intset{1}{f}\) and \(Q\) bounds the sum of their unused vertices by \(t-1\). Substitution gives
\[
 \ell_1-\eta_0
 =2w+c-1-u(\intset{1}{f})-u(Q)\ge w>0.
\]
For the second inequality, substitution directly gives
\[
 \eta_1-\ell_0
 =1+u(Q\setminus J)+u(\intset{1}{h-1})>0.
\]
The rank intervals therefore overlap. The matching edge at any common rank joins \(I\) to \(J\) and is deleted by \(R\).

\medskip\noindent\emph{Case 4: \(A\) and \(B\) are used, while \(I\) is unused.}
The matching has no left endpoint in \(I\), so we seek an edge of the triangle. Making all \(w\) vertices of \(I\) unused leaves
\(t-1-w=c-1\) unused vertices outside \(I\). In particular this case requires \(c\ge1\).

Let the first used vertex of \(Q\) be \(y_s=q+j\). There are \(a-c+u(Q)\) right endpoints before \(q\), and every vertex from \(q\) to \(q+j-1\) is unused. Consequently
\[
 s=a-c+u(Q)+1,\qquad 0\le j\le u(Q)\le c-1.
\]
At most \(c-1-u(Q)\) vertices of \(A\) are unused. Thus \(A\) contains at least \(a-c+u(Q)+1=s\) left endpoints, which places \(x_s\) within \(A\). Together with \(x_s\ge s\), this yields
\(a-c+1\le s\le x_s\le a\).

Write \(x_s=a-c+i\), with \(1\le i\le c\). The inequality \(x_s\ge s\) now gives \(u(Q)\le i-1\), and therefore \(0\le j<i\). This is exactly the defining condition for \(\edge{x_s}{y_s}\in\Theta_c\). Every possible case has forced a deleted edge, proving the lemma.
\end{proof}

\begin{proof}[Completion of the proof of Theorem~\ref{ext:thm:coordinated-construction}]\phantomsection\label{ext:proof:coordinated-construction}
Lemma~\ref{ext:lem:coordinated-blocker} proves the blocking property at \(N=2k+t-1\). We now count the four disjoint parts of \(D\), then distinguish the short and long deletions. This will give the three claimed edge counts before the extension to larger \(n\).

For each \(x\in A\), the set \(F_A\) has \(p-x\) edges starting at \(x\), while each \(y\in B\) contributes \(y-q\) edges to \(F_B\). Thus
\[
 |F_A|=ap-\frac{a(a+1)}2,\qquad
 |F_B|=b(N-q)-\frac{b(b-1)}2.
\]
Substituting the interval endpoints and using \(t=a+b\) gives
\[
 |F_A\cup F_B|=tk-ab-\binom{t+1}{2}+tc,\qquad
 |R|=wt,\qquad |\Theta_c|=\binom{c+1}{2}.
\]
The rectangle count is its number of cells, and the triangle count is \(1+\cdots+c\). Since \(w=t-c\), the total number of deletions simplifies to
\[
 |D|=tk-ab+\binom t2+\binom{c+1}{2}.
\]
On the other hand, direct substitution of \(N=2k+t-1\) and \(r=k-1\) gives
\(\binom N2-rN=tk+\binom t2\).
Subtracting \(|D|\) proves
\(e(K_N-D)=rN+ab-\binom{c+1}{2}=rN+\xi\).

Every edge in \(F_A\cup F_B\) is short: its length is at most \(p-1\le k-1\) in \(F_A\) and \(N-q\le k-1\) in \(F_B\). A rectangle edge \(\edge{h+u}{q+v}\), with \(0\le u<w\) and \(0\le v<t\), has length \(k+v-u\). It is short exactly when \(v<u\). Since \(w\le t\), the number of such cells is \(0+1+\cdots+(w-1)=\binom w2\). The other rectangle cells are long, and their number is
\[
 wt-\binom w2=\binom{t+1}{2}-\binom{c+1}{2}.
\]
Every triangle edge has length \(k+t+j-i\ge k+w\), so all its \(\binom{c+1}{2}\) edges are long. Altogether, exactly \(\binom{t+1}{2}\) long edges are deleted.

The complete graph has \(\binom{N-k+1}{2}\) edges of length at least \(k\), obtained by summing \(N-d\) over \(d=k,\ldots,N-1\). Therefore the number of retained long edges is
\[
 L(K_N-D)=\binom{N-k+1}{2}-\binom{t+1}{2}
          =rN-\frac{3r^2+r}{2}=L_0(N,r).
\]
Subtracting this from the total edge count gives \(S(K_N-D)=S_0(r)+\xi\). Finally, Lemma~\ref{ext:lem:ext-extension} adds \(r\) long edges for each further vertex while preserving the absence of a nonnested \(k\)-matching. The short count remains fixed, proving Eq.~\eqref{ext:eq:coordinated-counts} for every \(n\ge N\).
\end{proof}

\subsection{The optimized excess and its order}
\label{ext:subsec:coordinated-counts}
The construction gives an excess \(\xi=ab-\binom{c+1}{2}\), but the parameters compete for the available vertices and must satisfy \(k\ge3(a+b-c)-1\). We now maximize this explicit expression at fixed \(k,n\). First we take \(c=0\), where the triangle is absent and balancing \(a,b\) is enough. We then allow \(c>0\), obtain the largest excess available within the construction, and compare it with an upper bound valid for every avoiding graph.

\begin{corollary}\label{ext:lem:ext-core}
If \(a,b\ge1\), \(k\ge3(a+b)-1\), and \(n\ge2k+a+b-1\), there is
a graph with no nonnested \(k\)-matching and
\[
 e(G)=rn+ab,\qquad S(G)=S_0(r)+ab,\qquad L(G)=L_0(n,r).
\]
\end{corollary}
\begin{proof}
Set \(c=0\) in Theorem~\ref{ext:thm:coordinated-construction}.
Its parameter \(w\) then equals \(a+b\), giving the stated hypothesis
and all three counts.
\end{proof}

For \(c=0\), the total \(a+b\) is limited both by \(k\ge3(a+b)-1\) and by the required initial order \(2k+a+b-1\le n\). The largest permitted total is therefore
\begin{align}
 T(k,n)&=\min\left\{\left\lfloor\frac{k+1}{3}\right\rfloor,n-2k+1\right\}.
 \label{ext:eq:ext-T}
\end{align}
For a fixed total, the product \(ab\) is largest when the two integers differ by at most one. This gives
\begin{equation}\label{ext:eq:ext-quadratic-lower}
 f_k(n)\ge rn+\left\lfloor\frac{T(k,n)^2}{4}\right\rfloor
 \qquad(k\ge2,\ n\ge2k).
\end{equation}
When \(T(k,n)\ge2\), choose \(a=\lfloor T/2\rfloor\) and \(b=\lceil T/2\rceil\) in Corollary~\ref{ext:lem:ext-core}. When \(T(k,n)=1\), the extra term is zero, and the clique on \(2k-1\) vertices followed by Lemma~\ref{ext:lem:ext-extension} gives \(rn\) edges with counts \((S_0,L_0)\).

In particular, \(a=b=1\) recovers excess one at \(n=2k+1\) for \(k\ge5\). Allowing \(c>0\) makes it possible to use larger \(a+b\), because the length condition depends on \(w=a+b-c\). The next result determines exactly how much that flexibility gains within Theorem~\ref{ext:thm:coordinated-construction}.

\begin{corollary}\label{ext:cor:coordinated-envelope}
For \(k\ge5\), \(n\ge2k+1\), put
\[
 K=\left\lfloor\frac{k+1}{3}\right\rfloor,\qquad
 t_*=\min\{n-2k+1,2K-2\},\qquad c_*=\max\{0,t_*-K\}.
\]
Then the largest excess supplied by Theorem~\ref{ext:thm:coordinated-construction}
at these \(k,n\) is
\begin{equation}\label{ext:eq:coordinated-envelope}
 \left\lfloor\frac{t_*^2}{4}\right\rfloor-\binom{c_*+1}{2}.
\end{equation}
It equals \(\binom K2\) precisely when \(n\ge2k+2K-3\).
Every resulting graph satisfies \(L(G)=L_0(n,r)\), and thus
\begin{equation}\label{ext:eq:coordinated-plateau}
 f_k(n)\ge rn+\binom K2\qquad(n\ge2k+2K-3).
\end{equation}
\end{corollary}
\begin{proof}
Write \(K=\lfloor(k+1)/3\rfloor\), so the length constraint is \(t-c\le K\). For a fixed \(t=a+b\), balancing \(a,b\) maximizes their product, while choosing the smallest permitted \(c\) minimizes the triangular cost. The resulting choices are
\(a=\lfloor t/2\rfloor\), \(b=\lceil t/2\rceil\), and \(c=\max(0,t-K)\). They satisfy \(c\le\min(a,b)\) exactly when \(2\le t\le2K\).

For these choices, the excess becomes the single-variable function
\[
\phi_K(t)=\left\lfloor\frac{t^2}{4}\right\rfloor
 -\binom{\max\{0,t-K\}+1}{2}.
\]
To maximize it, compare consecutive values. For \(t<K\), increasing \(t\) by one adds \(\lfloor(t+1)/2\rfloor>0\). For \(K\le t<2K\), the increment is \(K-1-\lfloor t/2\rfloor\). Thus the function strictly increases until \(t=2K-2\), and its values at \(2K-2,2K-1,2K\) are equal.

The order restriction also requires \(t\le n-2k+1\). Hence
\(t_*=\min\{n-2k+1,2K-2\}\) attains the largest available excess, with \(c_*=\max(0,t_*-K)\), proving Eq.~\eqref{ext:eq:coordinated-envelope}. At the first maximum, \(t=2K-2\), the balanced choices are \(a=b=K-1\), \(c=K-2\). Their excess is \(\binom K2\), and this value is available precisely from order \(n=2k+2K-3\), proving Eq.~\eqref{ext:eq:coordinated-plateau}.
\end{proof}

The optimization raises the largest excess from \(\lfloor K^2/4\rfloor\) in the case \(c=0\) to \(\binom K2\) in the full construction. The former is available from \(n=2k+K-1\), while the latter is first attained at \(n=2k+2K-3\); Eq.~\eqref{ext:eq:coordinated-envelope} gives the intermediate values. For example, \(k=8\) gives \(K=3\). At \(n=17,18,19\), the best excesses in this construction are respectively \(1,2,3\), and the \(19\)-vertex example above attains the last value with \(a=b=2,c=1\).

Another way to see the limit of this family is the identity
\[
 ab-\binom{c+1}{2}
 =\binom w2-\binom{a-c}{2}-\binom{b-c}{2}.
\]
The last two terms are nonnegative and \(w\le K\), so the excess is at most \(\binom K2\). Since \(K=k/3+O(1)\), this maximum is \(k^2/18+O(k)\). The coefficient \(1/18\) describes the construction; to assess the possible extremal excess, we still need an upper bound applying to every graph.

We use the short-edge and long-edge decomposition of~\cite[Theorem~2.5]{BFT2025}. The short-edge bound follows by partitioning edges of each length into two matchings. The long-edge bound follows from the convex-geometric theorem of Capoyleas and Pach~\cite{CP1992}, after adding every possible short edge. We give both steps to explain the source of the universal bound.

\begin{proposition}\label{ext:prop:ext-upper-bounds}
For \(r\ge1\) and \(n\ge2r+1\), every ordered graph containing no
nonnested matching of size \(r+1\) satisfies
\begin{align}
 S(G)&\le2r^2,\qquad L(G)\le L_0(n,r),
 \label{ext:eq:ext-upper-components}\\
 e(G)&\le rn+\binom r2.
 \label{ext:eq:ext-upper-total}
\end{align}
\end{proposition}
\begin{proof}
Fix a length \(d\in\intset{1}{r}\). The potential edges \(i(i+d)\) form disjoint paths, one for each residue class modulo \(d\). Assign edges alternately along each path to two classes. Each class is a matching, because consecutive path edges are assigned differently, and different paths have disjoint vertex sets. Its edges all have the same length and hence cannot nest. Each class therefore contains at most \(r\) edges of \(G\), giving at most \(2r\) edges at this length. Summing over the \(r\) short lengths proves \(S(G)\le2r^2\).

For long edges, place the vertices on a convex polygon in their cyclic order. Any \(r+1\) pairwise crossing chords would form a nonnested matching, so the long-edge subgraph contains no such family. A short edge cannot belong to a family of \(r+1\) pairwise crossing chords: each of the other \(r\) chords would need a distinct endpoint strictly between its endpoints in the linear order, but that interval contains at most \(r-1\) vertices. Therefore adding every short edge still creates no \(r+1\) pairwise crossing chords.

There are
\(\sum_{d=1}^{r}(n-d)=rn-r(r+1)/2\)
potential short edges. Capoyleas and Pach~\cite{CP1992} bound the augmented graph by \(2rn-r(2r+1)\) edges. Subtracting the added short-edge count gives
\[
 L(G)\le2rn-r(2r+1)-rn+\frac{r(r+1)}2
       =rn-\frac{3r^2+r}{2}=L_0(n,r).
\]
Adding \(S(G)\le2r^2\) yields
\(e(G)\le rn+\binom r2\), proving the total bound.
\end{proof}

For \(r\ge4\), the parameter \(K\) equals \(\lceil r/3\rceil\), and
\(\binom K2\ge r(r-3)/18\).
Substituting \(k=r+1\) into Eq.~\eqref{ext:eq:coordinated-plateau} and combining it with Eq.~\eqref{ext:eq:ext-upper-total} gives
\begin{equation}\label{ext:eq:ext-uniform-excess}
 \frac{r(r-3)}{18}\le f_{r+1}(n)-rn\le\binom r2
 \qquad\left(n\ge2r+2\left\lceil\frac r3\right\rceil-1\right).
\end{equation}
Both sides grow proportionally to \(r^2\). Thus the excess
\(f_k(n)-(k-1)n\) is \(\Theta(k^2)\), meaning that it is bounded above and below by positive constant multiples of \(k^2\) as \(k\) grows, uniformly for
\(n\ge2k+2\lfloor(k+1)/3\rfloor-3\).
In particular, this holds throughout the simpler range \(n\ge3k\) stated in the abstract. The exact extremal excess and its best leading coefficient remain open in general.

There is also a consequence when \(k\), hence \(r=k-1\), is held fixed and only \(n\) grows. Lemma~\ref{ext:lem:ext-extension} makes the integer sequence \(f_{r+1}(n)-rn\) nondecreasing, because an optimal graph can always be extended with \(r\) further edges. Eq.~\eqref{ext:eq:ext-upper-total} bounds this sequence above by \(\binom r2\). A nondecreasing bounded sequence of integers eventually becomes constant, so \(f_{r+1}(n)=rn+C_r\) for some integer \(C_r\) and all sufficiently large \(n\). Determining that constant and an effective order from which it applies remains open in general.

\section{Conclusion}\label{sec:conclusion}
The deletion problem has a complete structural answer when the matching must use every vertex. For $k\ge2$, exactly $k$ deleted edges suffice to destroy all nonnested perfect matchings on $2k$ ordered vertices, and interval cuts classify all $2^k+k-2$ minimum deletion sets. The proof also produces an explicit avoiding matching whenever a set of at most $k$ deletions fails to be a blocker, and identifies a relation with blockers of pattern-avoiding permutation matrices.

Additional vertices allow a matching to avoid some positions altogether, so the minimum blockers on $2k$ vertices no longer describe the full extremal problem. Coordinated deletions at the two ends of the order produce graphs with more than $(k-1)n$ edges for every $k\ge5$ and $n\ge2k+1$. Throughout the range of Eq.~\eqref{eq:intro-quadratic}, the excess has order $k^2$, matching the order of the general upper bound. The optimization determines the best excess within this construction; the exact value of $f_k(n)$ remains open in general.

Two questions follow from these bounds. For each fixed $k$, the integer sequence $f_k(n)-(k-1)n$ eventually becomes constant, by the extension argument and the upper bound in Section~\ref{blk:sec:beyond}. Determining that constant and an effective order from which it applies remains open in general. When $k$ also grows, the best leading coefficient of the quadratic excess is unknown. Understanding how the excess increases as vertices are added beyond $2k$ may reveal constructions outside the coordinated endpoint families studied here.

\section*{Declarations}
This research received no specific grant from any funding agency in the public, commercial, or not-for-profit sectors. The author declares no competing interests.

\end{document}